\documentclass[11pt,a4paper]{article}
\usepackage{amsmath,amsthm,amssymb,amsfonts,amstext,graphics,dsfont,euscript,enumerate,multirow}
\usepackage[curve]{xypic}
\usepackage[latin1]{inputenc}
\usepackage[colorlinks=true,pdfstartview=FitV,linkcolor=blue,citecolor=blue,urlcolor=blue]{hyperref}
\usepackage[usenames,dvipsnames,svgnames,table]{xcolor}

\usepackage{layout}

\usepackage{makeidx}
\usepackage{url}

\usepackage[nottoc,numbib]{tocbibind}

\theoremstyle{definition}

\newcommand{\mb}{\mathds}

\newcommand{\mf}{\mathfrak}
\newcommand{\mc}{\mathcal}
\newcommand{\ra}{\rightarrow}
\newcommand{\lra}{\longrightarrow}

\newcommand{\A}{\mathbf{A}}
\newcommand{\B}{\mathbf{B}}

\newcommand{\D}{\mathbf{D}}

\newcommand{\Dt}{\widetilde{\mathbf{D}}}
\newcommand{\E}{\mathbf{E}}

\newcommand{\N}{\mathbf{N}}
\newcommand{\W}{\mathbf{W}}

\newcommand{\TD}{\widetilde{D}}
\newcommand{\TDlog}{\widetilde{D}_{\textrm{log}}}

\newcommand{\Fil}{\textrm{Fil}}

\newcommand{\Be}{\B_e}
\newcommand{\We}{\W_e}
\newcommand{\Wpluse}{\W^+_e}
\newcommand{\Wdr}{\W_{\textrm{dR}}}
\newcommand{\Wplusdr}{\W^+_{\textrm{dR}}}

\newcommand{\Dplusdifn}{\D^+_{\mbox{\scriptsize dif}, n}}
\newcommand{\Dplusdifneins}{\D^+_{\mbox{\scriptsize dif}, n+1}}
\newcommand{\Ddifn}{\D_{\mbox{\scriptsize dif}, n}}
\newcommand{\Ddifneins}{\D_{\mbox{\scriptsize dif}, n+1}}
\newcommand{\Dplusdif}{\D^+_{\mbox{\scriptsize dif}}}
\newcommand{\Ddif}{\D_{\mbox{\scriptsize dif}}}

\newcommand{\Dcris}{\D_{\mbox{\scriptsize cris}}}

\newcommand{\Dst}{\D_{\mbox{\scriptsize st}}}

\newcommand{\Dplusst}{\D^+_{\mbox{\scriptsize st}}}
\newcommand{\Ddr}{\D_{\mbox{\scriptsize dR}}}

\newcommand{\Ddagger}{\D^\dagger}

\newcommand{\Ddaggerrig}{\D_{\textrm{rig}}^\dagger}

\newcommand{\Ddaggerrigk}{\D_{\textrm{rig},K}^\dagger}

\newcommand{\Dinftye}{\D_{\infty, e}}
\newcommand{\Dinftyf}{\D_{\infty, f}}
\newcommand{\Dinftyg}{\D_{\infty, g}}
\newcommand{\Dinftyast}{\D_{\infty, *}}

\newcommand{\Ndr}{\N_{\textrm{dR}}}
\newcommand{\Hiw}{H^1_{\textrm{Iw}}}

\newcommand{\Bcris}{\B_{\mbox{\scriptsize cris}}}
\newcommand{\Bst}{\B_{\mbox{\scriptsize st}}}
\newcommand{\Bdr}{\B_{\mbox{\scriptsize dR}}}
\newcommand{\Bplusdr}{\B_{\mbox{\scriptsize dR}}^+}
\newcommand{\Bmax}{\B_{\mbox{\scriptsize max}}}
\newcommand{\Bplusmax}{\B_{\mbox{\scriptsize max}}^+}
\newcommand{\Bpluscris}{\B_{\mbox{\scriptsize cris}}^+}

\newcommand{\Btrig}{\widetilde{\B}_{\mbox{\scriptsize rig}}}

\newcommand{\Et}{\widetilde{\E}}
\newcommand{\Etplus}{\widetilde{\E}^+}
\newcommand{\At}{\widetilde{\A}}
\newcommand{\Atplus}{\widetilde{\A}^+}
\newcommand{\Aplus}{\A^+}
\newcommand{\Bt}{\widetilde{\B}}
\newcommand{\Btplus}{\widetilde{\B}^+}

\newcommand{\Bk}{\B_K}
\newcommand{\Ak}{\A_K}

\newcommand{\Ek}{\E_K}
\newcommand{\Eplusk}{\E^+_K}
\newcommand{\Ef}{\E_F}

\newcommand{\Aqp}{\A_{\mb{Q}_p}}
\newcommand{\Bqp}{\B_{\mb{Q}_p}}
\newcommand{\Eqp}{\E_{\mb{Q}_p}}
\newcommand{\Eplus}{\E^+}

\newcommand{\Btdagger}{\widetilde{\B}^\dagger}
\newcommand{\Btdaggerr}{\widetilde{\B}^{\dagger, r}}
\newcommand{\Atdagger}{\widetilde{\A}^\dagger}
\newcommand{\Atdaggerr}{\widetilde{\A}^{\dagger, r}}
\newcommand{\Btdaggern}{\widetilde{\B}^{\dagger, n}}
\newcommand{\Btdaggerrig}{\widetilde{\B}^\dagger_{\textrm{rig}}}
\newcommand{\Btdaggerrigk}{\widetilde{\B}^\dagger_{\textrm{rig}, K}}
\newcommand{\Atdaggerrig}{\widetilde{\A}^\dagger_{\mbox{\scriptsize rig}}}
\newcommand{\Btdaggerlog}{\widetilde{\B}^\dagger_{\mbox{\scriptsize log}}}

\newcommand{\Btplusrig}{\widetilde{\B}^+_{\mbox{\scriptsize rig}}}
\newcommand{\Btpluslog}{\widetilde{\B}^+_{\mbox{\scriptsize log}}}
\newcommand{\Btdaggerrrig}{\widetilde{\B}^{\dagger, r}_{\mbox{\scriptsize rig}}}
\newcommand{\Btdaggerrnrig}{\widetilde{\B}^{\dagger, r_n}_{\textrm{rig}}}

\newcommand{\Btdaggerrk}{\widetilde{\B}^{\dagger,r}_K}
\newcommand{\Atdaggerrk}{\widetilde{\A}^{\dagger,r}_K}
\newcommand{\Adaggerrk}{\A^{\dagger,r}_K}

\newcommand{\Bdagger}{\B^{\dagger}}
\newcommand{\Adagger}{\A^\dagger}
\newcommand{\Adaggerr}{\A^{\dagger,r}}
\newcommand{\Bdaggerr}{\B^{\dagger, r}}
\newcommand{\Bdaggerk}{\B^{\dagger}_K}
\newcommand{\Bdaggerf}{\B^{\dagger}_F}

\newcommand{\Bdaggerrk}{\B^{\dagger, r}_K}

\newcommand{\Bdaggerrigk}{\B^{\dagger}_{\mbox{\scriptsize rig}, K}}
\newcommand{\Bdaggerrigl}{\B^{\dagger}_{\mbox{\scriptsize rig}, L}}
\newcommand{\Bdaggerrigf}{\B^{\dagger}_{\mbox{\scriptsize rig}, F}}

\newcommand{\Bdaggerrig}{\B^{\dagger}_{\mbox{\scriptsize rig}}}
\newcommand{\Bdaggerlog}{\B^{\dagger}_{\mbox{\scriptsize log}}}

\newcommand{\Bdaggerrnrigk}{\B^{\dagger, r_n}_{\textrm{rig}, K}}
\newcommand{\Bdaggerrndrigk}{\B^{\dagger, r_{n(D)}}_{\mbox{\scriptsize rig}, K}}
\newcommand{\Bdaggerrneinsrigk}{\B^{\dagger, r_{n + 1}}_{\mbox{\scriptsize rig}, K}}

\newcommand{\Bdaggerlogk}{\B^{\dagger}_{\mbox{\scriptsize log}, K}}
\newcommand{\Bdaggerlogl}{\B^{\dagger}_{\mbox{\scriptsize log}, L}}

\newcommand{\Bdaggernlogk}{\B^{\dagger, n}_{\mbox{\scriptsize log}, K}}

\makeatletter

\makeindex

\newtheorem{thm}{Theorem}[section]
\newtheorem*{thm*}{Theorem}
\newtheorem{cor}[thm]{Corollary}
\newtheorem{prop}[thm]{Proposition}
\newtheorem{lem}[thm]{Lemma}
\newtheorem{rem}[thm]{Remark}
\newtheorem{defi}[thm]{Definition}

\let\tmp\oddsidemargin
\let\oddsidemargin\evensidemargin
\let\evensidemargin\tmp
\reversemarginpar

\makeatletter
\def\blfootnote{\gdef\@thefnmark{}\@footnotetext}
\makeatother

\title{On Perrin-Riou's exponential map for $(\varphi, \Gamma)$-modules}

\author{Andreas Riedel}

\date{November 1, 2014}

\begin{document}

\maketitle

\begin{abstract}

Let $K / \mb{Q}_p$ be a finite Galois extension and $D$ a 
$(\varphi, \Gamma)$-module over the Robba-ring $\Bdaggerrigk$.
We give a generalization of the Bloch-Kato exponential
map for $D$ using continuous Galois-cohomology groups $H^i(G_K, \W(D))$
for the $B$-pair $\W(D)$ associated to $D$. We construct a big exponential map 
$\Omega_{D,h}$ ($h \in \mb{N}$)
for cyclotomic extensions of $K$ for $D$ in the style of Perrin-Riou 
using the theory of Berger's $B$-pairs, which interpolates the generalized 
Bloch-Kato exponential maps on the finite levels.

\end{abstract}

\thispagestyle{empty}

\normalsize
\tableofcontents


\section{Introduction}
%

\blfootnote{2010 Mathematical Subject Classification: 11F80 (primary), 11F85, 
11S25 (secondary). Keywords: $p$-adic Hodge theory, $B$-pairs, big exponential 
map, $(\varphi, \Gamma)$-modules.}

We fix some notation. Let $K$ be a
finite extension of $\mb{Q}_p$ and denote by $F$ the biggest
subextension of $K$ that is unramified over $\mb{Q}_p$. Let
$\mu_{p^n}$ denote the roots of unity in a fixed algebraic closure
$\overline{K}$ of $K$ and set $K_n = K(\mu_{p^n})$ and $K_\infty =
\bigcup_n K_n$.  As usual $G_K$ denotes the absolute Galois group of
$K$, and we set $H_K = \textrm{Gal}(\overline{K} / K_\infty)$ and
$\Gamma_K = G_K / H_K$. Perrin-Riou considers a
distribution algebra $\mc{H}(\Gamma_K)$ that contains the usual 
Iwasawa algebra $\Lambda(\Gamma_K)$.


Recall that by the theory of Fontaine one 
may then associate to any $p$-adic representation of $V$ of $G_K$ finite 
dimensional $F$-vector spaces
\[
  \Dcris(V)  \subset \Dst(V) \subset \Ddr(V),
\] 
via the $\mb{Q}_p$-algebras $\Bcris, \Bst, \Bdr$,
where the first two come equipped with an action of a Frobenius
$\varphi$ and a nilpotent monodromy operator $N$, and the third one
is equipped with a filtration.

Bloch and Kato constructed
the exponential map $\exp: \Ddr(V) \lra H^1(K, V)$, which is nothing
but a transition morphism arising from a long exact sequence
of continuous Galois cohomology. They showed
that there exists a deep connection between this map and the special 
values of the complex $L$-function attached to $V$.

Perrin-Riou set out to adapt this construction to the theory of
$p$-adic $L$-functions.  Explicitly, for $K / \mb{Q}_p$ unramified
and $V$ crystalline (i.e. $\dim_F \Dcris(V) = \dim_{\mb{Q}_p} V$) she
constructed a map $\Omega_{V(j), h}$ that fits into the following
diagram
\begin{equation}
\label{introinter}
  \xymatrix{ 
    \mc{H}(\Gamma_K) \otimes_{\mb{Q}_p} \Dcris(V(j)) \ar[r]^-{\Omega_{V(j), h}} 
    \ar[d]^{\Xi_{n, j}} & \mc{H}(\Gamma_K)
    \otimes_{\Lambda} \Hiw(K, V(j)) / V(j)^{G_{\mb{Q}_{p, n}}} \ar[d]^{\textrm{pr}_n} \\
    K_n \otimes \Dst(V(j)) \ar[r]^{(h-1)! \exp_{K_nn, V(j)}} & 
    H^1(K_n, V(j))
  }
\end{equation}
for $h \gg 0$, $j \gg 0$ and all $n$, where $\Xi_{n,j}$ and $\textrm{pr}_n$
are certain canonical projections and
$\Hiw$ denotes Iwasawa cohomology with respect to the tower $(K_n)_n$. 
The point here is that $\Omega_{V, h}$
interpolates infinitely many Bloch-Kato exponential maps on the finite
levels.

In \cite{perrin02}, Perrin-Riou extended her construction to
semi-stable representations over unramified extensions. She gave a
definition of a free $\mc{H}(\Gamma_K)$-module $\Dinftyg(V)$ 
and a map
\[
  \Omega_{V, h}: \Dinftyg(V) \lra \mc{H}(\Gamma_{\mb{Q}_p}) \otimes_\Lambda
  \Hiw(K, V) / V^{G_{K_\infty}} 
\]
that has a similar interpolation property as (\ref{introinter}) 
for $j \gg 0$ and $n \gg 0$. 

It was Berger who gave an explicit description of a ``big exponential
map'' for crystalline representations using these modules not only on
the finite level, but on the whole of $\mc{H}(\Gamma_K) \otimes
\Dcris(V)$ and $\Hiw(K, V)$. His fundamental insight is a
comparison isomorphism depending on 
the construction of another ring $\Bdaggerrigk$ resp. 
$\Bdaggerlogk$.

Berger considered in the crystalline case
the element $\nabla_{h-1} \circ \ldots \circ \nabla_0$, where
$\nabla_i \in \mc{H}(\Gamma_K)$ is Perrin-Riou's differential
operator, and showed that one obtains a map
\begin{align*}
  \nabla_{h-1} \circ \ldots \circ \nabla_0: 
  (\varphi - 1)(\Bdaggerrigk \otimes \Dcris(V(j)))^{\psi = 1}  \lra& 
  (\varphi - 1)\Ddaggerrig(V(j))^{\psi = 1} \\
  &= \mc{H}(\Gamma_{\mb{Q}_p})
  \otimes_\Lambda \Hiw(K, V(j)) / V(j)^{G_{K_\infty}}
\end{align*}
that actually coincides with Perrin-Riou's $\Omega_{V(j), h}$ 
(see \cite{berger03}, Theorem II.13).

Since one has an embedding of the category of $p$-adic representations
into the category of all $(\varphi, \Gamma)$-modules over
$\Bdaggerrigk$ via the functor $\Ddaggerrig(~)$, one might be inclined
to generalize the framework of exponential maps to this setting.  
Similarly as in the \'etale case, one may define
finite-dimensional vector spaces $\Dcris(D)$, $\Dst(D)$ and $\Ddr(D)$,
generalized Bloch-Kato exponential maps
\[
  \exp: \Ddr(D) \ra H^1(K,D),
\]
and develop the notion of a $(\varphi, \Gamma)$-module being
crystalline, semi-stable or de Rham. We define a
$\mc{H}(\Gamma_K)$-module $\Dinftyg(D)$ and show that there exists a
map for $h \gg 0$
\[
  \Omega_{D, h} := \nabla_{h-1} \circ \ldots \circ \nabla_{0}: 
  \Dinftyg(D) \lra (\varphi - 1) D^{\psi = 1}.
\]
The main result of the third section is then the following interpolation 
property (see Theorem \ref{thmsecond} for the precise statement):
\begin{thm*}
  Let $D$ be a de Rham $(\varphi, \Gamma_K)$-module over
  $\Bdaggerrigk$, $g \in \Dinftyg(D)$ and $G$ a ``complete solution''
  (cf. Definition \ref{deficomplesol})
  for $g$ in $L$ and let $h \gg 0$.  Then for $k \geq 1 - h$ and $n
  \gg 1$ one has
  \begin{align*}
    h^1_{K_n, D(k)}( &\nabla_{h-1} \circ \ldots \circ \nabla_0(g)
    \otimes e_k) \\
    &= p^{-n(K_n)} (-1)^{h+k-1}(h+1-k)!
    \frac{1}{[L_n:K_n]}\textrm{Cor}_{L_n/K_n}\exp_{K_n, D(k)}(\Xi_{n,
      k}(G)).
  \end{align*}
\end{thm*}

If one is interested in the construction of $p$-adic $L$-functions,
one needs to construct a certain ``inverse'' of the map $\Omega_h$.
This construction depends on the so-called reciprocity law for 
$(\varphi, \Gamma)$-modules, 
which we will return to in a future paper, using the results of this article.

%
%

We remark that during this work learned of the results of K.
Nakamura \cite{naka12}, who gave a description of a ``big exponential map''
for $(\varphi, \Gamma)$-modules. We briefly outline how our constructions differ
from \cite{naka12}. Firstly, we show the existence of a fundamental exact 
sequence
\[
    0 \lra X^0(\TD) \lra \TDlog[1/t] \lra
    X \lra X^1(\TD) \lra 0 
\]
of continuous $G_K$-modules associated to 
any $(\varphi, \Gamma)$-module $D$, generalizing the Bloch-Kato 
fundamental exact sequence (cf. p. \pageref{setx} for the definition of $X$). 
Taking continuous Galois-cohomology one
obtains, in a completely analoguous fashion to the \'etale case,
a generalized Bloch-Kato exponential map as the transition map for 
cohomology, which is automatically functorial by construction.

Secondly, we introduce certain finitely generated $\mc{H}(\Gamma)$-submodules
$\D_{\infty, *}(D)$ of the free $\mc{B}(\Gamma)$-module $\Ndr(D)^{\psi = 0}$ such that 
$X$ arises in a natural way after projecting to some finite level $K_n$ and 
looking at the Bloch-Kato exponential map on this level. Using these two different 
ingredients we are able to show the main theorem above.

Some important facts about these modules are:
\begin{itemize}
\item the $\D_{\infty, *}(D)$ are invariant under Tate-twists (as opposed to 
  $(1 - \varphi) \Ndr(D)^{\psi = 1}$), and
\item the $\D_{\infty, *}(D)$ remove the ambiguity in the statements \cite{naka12}, 
  Theorem 3.10 (1) and \cite{berger03}, Theorem II.16 about the existence of an 
  element $y$ such that $(1 - \varphi)(y) = x$.
\end{itemize}
These points and further examples suggest that, in order to study 
reciprocity laws and the connection of exponential maps with $p$-adic $L$-functions,
one should look at these modules instead of $(1-\varphi)\Ndr(D)^{\psi = 1}$.
We also refer to the introduction of \cite{perrin01} in the \'etale unramified case for some
further motivation.

\vspace{5mm}

\noindent \textbf{Acknowledgements}. 
This article is based on a part of my thesis, and I would like to thank my 
advisor Otmar Venjakob for his encouragement, patience and advice.

\section{Rings and Modules}


\subsection{General notations}

The general strategy of Fontaine is to study $p$-adic representations
by certain \textit{admissibility} conditions. Recall that if $V$ is a
finite dimensional $\mb{Q}_p$-vectorspace endowed with a continuous
action of a topological group $G$ and if $B$ is a topological
$\mb{Q}_p$-algebra which also carries an action of $G$, then Fontaine
considers the $B^G$-modules $D_B(V) = (B \otimes_{\mb{Q}_p} V)^G$. It
inherits actions from $B$ and $V$. One says that $V$ is
\textbf{$B$-admissible} if $B \otimes_{\mb{Q}_p} V \cong B^d$ as
$G$-modules.

Let $k$ be a perfect field of characteristic $p$. We denote by $W(k)$
the ring of Witt-vectors for $k$ and set $F =
\textrm{Quot}(W(k))$. Let $K/F$ be a totally ramified extension of
$F$.  Fix an algebraic closure $\overline{F}$ of $F$ and denote by
$\mb{C}_p = \widehat{\overline{F}}$ the $p$-adic completion of this
closure. Let $G_K = \mbox{Gal}(\overline{K} / K)$\index{GK@$G_K$} be the
group of automorphisms of $\overline{K}$ which fix $K$. By continuity
these are also the $K$-linear automorphisms of $\mb{C}_p$. Let
$\mc{O}_{\mb{C}_p}$ be the ring of integers of $\mb{C}_p$ and
$\mf{m}_{\mb{C}_p}$ its maximal ideal. We have $\mc{O}_{\mb{C}_p} /
\mf{m}_{\mb{C}_p} = \overline{k}$.

We denote by $\mu_{p^n}$ the group of roots of unity of $p^n$-order in
$\mb{C}_p$ and set $K_n = K(\mu_{p^n})$\index{Kn@$K_n$}. Further we
pose $K_\infty = \bigcup_n K_n$\index{Kinfty@$K_\infty$}.  We fix once
and for all a compatible set of primitive $p$-th roots of unity $\{
\zeta_{p^n} \}_{n \geq 0}$ such that $\zeta_1 = 1$, $\zeta_p \not= 1$,
$\zeta^p_{p^{n+1}} = \zeta_{p^n}$.  One has the cyclotomic character
$\chi: G_K \ra \mb{Z}_p^\times$ which is defined by the formula
$g(\zeta_{p^n}) = \zeta^{\chi(g)}_{p^n}$ for $n \geq 1$ and $g \in
G_K$.We set $H_K = \ker(\chi)$\index{HK@$H_K$} and $\Gamma_K = G_K /
H_K$\index{GammaK@$\Gamma_K$}, which is the Galois group 
of $K_\infty / K$. We know that this can also be identified via the cyclotomic
character with an open subgroup of $\mb{Z}_p^\times$.

If $K / \mb{Q}_p$ is a finite extension denote by $F =
K_0$\index{Kzero@$K_0$} the maximal unramified extension of $\mb{Q}_p$
in $K$.  Further denote by $K_0'$\index{Kzeroprime@$K_0'$} the biggest
unramified subextension of $K_0$ in $K_\infty$.

By a $p$-adic representation we mean a finite dimensional
$\mb{Q}_p$-vectorspace endowed with a continuous and linear action of
$G_K$. A $\mb{Z}_p$-representation is a free $\mb{Z}_p$-module of
finite rank equipped with a linear and continuous action of $G_K$. It
is known that if $V$ is a $p$-adic representation then there exists a
$\mb{Z}_p$-lattice $T$ in $V$ that is stable under the action of
$G_K$.

If $C^\bullet(-)$ denotes complex of $R$-modules for some commutative ring 
(for example, $C^\bullet(G_K, M)$) $R$ we denote as usual $R\Gamma(-)$
the complex which we regard as an object in the derived category of
$R$-modules.


\subsection{Rings in $p$-adic Hodge theory}

We first recall certain rings constructed by Fontaine, see for instance 
\cite{fontouy}. Let
\[
	\Et = \varprojlim_{x \mapsto x^p} \mb{C}_p = \{ (x^{(0)},
        x^{(1)}, \ldots) | ~ x^{(i)} \in \mb{C}_p, (x^{(i+1)})^p =
        x^{(i)} ~\forall i \}.
\]\index{Etilde@$\Et$}
Similarly, let
\begin{align*}
	\Etplus = \varprojlim_{x \mapsto x^p} \mc{O}_{\mb{C}_p} &= \{
        (x^{(0)}, x^{(1)}, \ldots) | ~ x^{(i)} \in \mc{O}_{\mb{C}_p},
        (x^{(i+1)})^p = x^{(i)} ~\forall i \} \\ &\cong \{ (x_n)_{n \in
          \mb{N}} |~ x_n \in \mc{O}_{\mb{C}_p} / p \mc{O}_{\mb{C}_p},
        x^p_{n+1} = x_n ~ \forall n \}. \index{Etildeplus@$\Etplus$}
\end{align*}
This is the set of elements of $\Et$ 
such that $x^{(0)} \in
\mc{O}_{\mb{C}_p}$.  One can define multiplication and addition on
these sets. Also, one knows that $\Et$ is the fraction field of $\Etplus$.

With the choices of the primitive $p^n$-th roots of unity one defines
the elements $\varepsilon = (1, \zeta_p, \ldots ) \in 
\Etplus$\index{epsilon@$\varepsilon$} and $\overline{\pi} = \varepsilon - 1 
\in
\Etplus$\index{pioverline@$\overline{\pi}$}.  One has the usual 
commuting actions of a Frobenius $\varphi$ and the 
Galois group $G_{\mb{Q}_p}$ on $\Et$, which restrict to actions of $\Etplus$.
For $K / \mb{Q}_p$ finite we set 
\[
	\Eplusk = \{ (x_n) \in \Etplus|~ x_n \in \mc{O}_{K_n} / p \mc{O}_{K_n} ~
	\forall n \geq n(K) \},
\]\index{Eplusk@$\Eplusk$}
where $n(K)$\index{nK@$n(K)$} is some constant depending on $K$ which
arises in the fields of norm theory of Fontaine-Wintenberger (cf.
\cite{fonwinten}). We put $\Ek = \Eplusk[1 /
  \overline{\pi}]$\index{Ek@$\Ek$}. One can show that that $\Ef =
\kappa((\overline{\pi}))$\index{Ef@$\Ef$} and one defines
$\E$\index{E@$\E$} as the seperable closure of $\Ef$ in $\Et$. Let
$\Eplus = \E \cap \Etplus$\index{Eplus@$\Eplus$} and $\mf{m}_E = \E
\cap \mf{m}_{\Et}$\index{mE@$\mf{m}_E$}. One can show that $\Ek =
\E^{H_K}$ and one knows that $\mbox{Gal}(\E / \Ek) \cong H_K$.
%


\label{secringscharnull}

Let $W$ be the Witt functor. We set
\[
	\Atplus = W(\Etplus),~~~ \At = W(\Et) = W(\mbox{Frac}(\Etplus)),~~~
	\Btplus = \Atplus [1/p]. \index{Atildeplus@$\Atplus$} 
\index{Atilde@$\At$} \index{Btildeplus@$\Btplus$}
\]
We write elements $x \in \Btplus$ as $x = \sum_{k \gg -\infty}^\infty p^k [x_k]$
where $x_k \in \Etplus$ and $[x_k]$ is its Teichm\"uller
representative. The commuting actions of $\varphi$ and $G_{\mb{Q}_p}$
on $\Etplus$ extend to an action of $\Btplus$ (and $\At, \Bt,
\ldots$).\index{phi@$\varphi$}

We have a ring homomorphism
\[
	\theta: \Btplus \longrightarrow \mb{C}_p,~~  \sum_{k \gg -\infty}^\infty p^k 
        [x_k] \longmapsto 		  \sum_{k \gg -\infty}^\infty p^k x_k^{(0)}.
\] \index{theta@$\theta$}
We set $\pi = [\overline{\pi}] = [\varepsilon] - 1$\index{pi@$\pi$},
$\pi_n = [\varepsilon^{p^{-n}}] -1$\index{pin@$\pi_n$}, $\omega = \pi
/ \pi_1$\index{omega@$\omega$} and $q = \varphi(\omega) = \varphi(\pi)
/ \pi$\index{q@$q$}. Then $\ker(\theta)$ is a principal ideal generated by
$\omega$.

The ring $\Bplusdr$ is defined by completing $\Btplus$ with the 
$\ker(\theta)$-adic topology, i.e.,
$\Bplusdr = \varprojlim_{n \geq 0} \Btplus / (\ker(\theta)^n)$.
This gives a complete discrete valuation ring with maximal ideal
$\ker(\theta)$.  One can show that $\log([\varepsilon])$ converges in
$\Bplusdr$, and we denote this element by $t$\index{t@$t$}.  It is a
generator of the maximal ideal, hence we can form the field $\Bdr =
\Bplusdr[1/t]$\index{Bdr@$\Bdr$}. This field is equipped with an
action of $G_{\mb{Q}_p}$ and a canonical filtration defined by
$\mbox{Fil}^i (\Bdr) = t^i \Bplusdr,~ i \in \mb{Z}$.

We say that a $p$-adic representation $V$ of $G_K$ is \textbf{de Rham}
if it is $\Bdr$-admissible. We put
\[
	\Ddr(V) = (\Bdr \otimes_{\mb{Q}_p} V)^{G_K},~~~ 
	\mbox{Fil}^i \Ddr(V) = (\mbox{Fil}^i \Bdr \otimes_{\mb{Q}_p} V)^{G_K}.
\]
From Fontaine's theory it is known that $\Ddr(V)$ is finite dimensional  
$K$-vectorspace which we endow with the above (exhaustive, seperated and 
decreasing) filtration. 

We say that a $p$-adic representation $V$ is \textbf{Hodge-Tate} with
Hodge-Tate weights $h_1, \ldots, h_d$ if one has a decomposition
$\mb{C}_p \otimes_{\mb{Q}_p} V \cong \bigoplus_{i = 1}^d
\mb{C}_p(h_i)$. We say that $V$ is \textbf{positive} if its Hodge-Tate
weights are negative. It is known that every de Rham representation is
Hodge-Tate and that the Hodge-Tate weights are those integers $h$ such
that there is a jump in the filtration at $-h$, i.e.  $\mbox{Fil}^{-h}
\Ddr(V) \not= \mbox{Fil}^{-h+1} \Ddr(V)$. With this
convention the representation $\mb{Q}_p(1)$ is of weight $1$. 

Let
\[
	\Aqp = \widehat{\mb{Z}[[\pi]][1/\pi]} = \left\{ \sum_{k \in
          \mb{Z}} a_k \pi^k \left| ~ a_k \in \mb{Z}_p, \lim_{k \ra -
          \infty} v_p(a_k) = +\infty \right. \right\} \hookrightarrow
        \At, \index{Aqp@$\Aqp$}
\]
and set $\Bqp = \Aqp[1/p]$\index{Bqp@$\Bqp$}. Then $\Bqp$ is a field,
complete for the $p$-adic valuation with ring of integers $\Aqp$ and
residue field $\Eqp$.  Let $\B$\index{B@$\B$} be the $p$-adic completion 
of the maximal unramified extension of $\Bqp$ in $\Bt$. We define $\A = \B \cap 
\At$\index{A@$\A$}, $\Aplus
= \A \cap \Atplus$\index{Aplus@$\Aplus$}. These rings still have the
commuting action of $\varphi$ and $G_{\mb{Q}_p}$. We put $\Ak =
\A^{H_K}$\index{Ak@$\Ak$} and $\Bk = \Ak[1/p]$\index{Bk@$\Bk$}.  By
Hensel's Lemma there exists a unique lift $\pi_K \in \Ak$
\index{piK@$\pi_K$} such that the reduction mod $p$ is equal to
$\overline{\pi}_K$, viewed as an element in $\At$.

Colmez has defined the ring 
\[
\Bplusmax = \{ \sum_{n \geq 0} a_n { \omega^n \over p^n}|~ a_n \in
\Btplus,~ a_n \ra 0 \mbox{ for } n \ra \infty \}
\]\index{Bplusmax@$\Bplusmax$}
which is "very close" to $\Bpluscris$. We set $\Bmax = \Bplusmax
[1/t]$ \index{Bmax@$\Bmax$}. There is a canonical injection of $\Bmax$
into $\Bdr$ and it is therefore equipped with a canonical
filtration. There are actions of $\varphi$ and $G_{\mb{Q}_p}$
on $\Bmax$, which extend the actions on $\Atplus \ra
\Atplus$. Let
\[
	\Btplusrig = \bigcap_{n=0}^\infty \varphi^n(\Bplusmax)
\index{Btildeplusrig@$\Btplusrig$}
\]
and $\Btrig = \Btplusrig [1/t]$. We remark that one has
$\Btrig = \bigcap_{n=0}^\infty \varphi^n(\Bcris)$
and hence in particular $\Btrig^{\varphi = 1} = \Bmax^{\varphi=1} = \Bcris^{\varphi=1}$.
We say that a representation is \textbf{crystalline} if it is 
$\Bmax$-admissible, which is the same as asking that it be
$\Btplusrig[1/t]$-admissible. We put
\[
	\Dcris(V) = (\Bmax \otimes_{\mb{Q}_p} V)^{G_K} = 
	(\Btplusrig[1/t] \otimes_{\mb{Q}_p} V)^{G_K}.
\]\index{DcrisV@$\Dcris(V)$}
This is a $K_0$-vectorspace of dimension $d$, equipped with a
filtration induced by $\Bdr$ and an action of Frobenius induced by
$\Bmax$. If $V$ is crystalline we have $\Ddr(V) = K
\otimes_{K_0} \Dcris(V)$ which shows that a crystalline representation
is also de Rham.

Following Berger the series $\log (\overline{\pi}^{(0)}) + \log
(\overline{\pi} / \overline{\pi}^{(0)})$, after a choice of $\log p$,
converges in $\Bplusdr$, and we denote the limit by $\log
[\overline{\pi}] $. This element is transcendent over
$\mbox{Frac}(\Bplusmax)$, and we set $\Bst = \Bmax[ \log [
    \overline{\pi} ] ]$ and $\Btpluslog = \Btplusrig [ \log [
    \overline{\pi} ]]$. We say that a representation is
\textbf{semistable} if it is $\Bst$-admissible, which is the same as
asking it being $\Btpluslog [1/t]$-admissible. Similarly, as in the
crystalline case we put
\[
	\Dst(V) = (\Bst \otimes_{\mb{Q}_p} V)^{G_K} = (\Btpluslog[1/t] 
	\otimes_{\mb{Q}_p} V)^{G_K}.
\]\index{DstV@$\Dst(V)$}
Again this is a $K_0$-vectorspace of dimension $d$, equipped with a
filtration and an action of Frobenius induced by $\Bst$. As before we
have in this case $\Ddr(V) = K \otimes_{K_0} \Dst(V)$. Additionally
one can define the monodromy operator $N = - d/ d\log [ \overline{\pi}
]$ on $\Bst$ which induces a nilpotent endomorphism on $\Dst(V)$ and
satisfies the relation $N \varphi = p \varphi N$. We also make use of
the finite dimensional $K_0$-vectorrspace $\Dplusst(V) = (\Btpluslog
\otimes_{\mb{Q}_p} V)^{G_K}$.

Recall that elements $x \in \Bt$ may be written in the form 
$x = \sum_{k \gg -\infty} p^k [x_k]$ with $x_k \in \Et$. For $r > 0$ we set
\[
	\Btdaggerr = \left\{ x \in \Bt \left| ~ \lim_{k \ra +\infty} v_{\E}(x_k) +
	\frac{pr}{p-1} k = + \infty
	\right. \right\}.
\]\index{Btildedaggerr@$\Btdaggerr$}
We note that $x$ as above converges in $\Bdr$ if and only if $\sum_{k
  \gg -\infty} p^k x_k^{(0)}$ converges in $\mb{C}_p$. 

For $n \geq 0$ we set once and for all $r_n = (p-1)p^{n-1}$.
Colmez and Cherbonnier showed that for $n$ big enough such that $r_n
\geq r$ there is an injection
\[
	\iota_n = \varphi^{-n}: \Btdaggerr \longrightarrow \Bplusdr,~~~
	\sum_{k \gg -\infty} p^k [x_k] \longmapsto \sum_{k \gg -\infty} p^k [x_k^{p^{-n}}].
\]\index{iotan@$\iota_n$}
We put $\Btdaggern = \widetilde{\B}^{\dagger, r_n}$. Let $\Bdaggerr =
\B \cap \Btdaggerr$, $\Btdagger = \bigcup_{r \geq 0} \Btdaggerr$,
$\Bdagger = \bigcup_{r \geq 0} \Bdaggerr$. Let $\Atdaggerr$ be the
elements of $\Btdaggerr \cap \At$ such that $v_E(x) + \frac{pr}{p-1} k
\geq 0$ for all $k \geq 0$.  Let $\Adaggerr = \Atdaggerr \cap \A$,
$\Adagger = \Atdagger \cap \A$, $\Atdagger = \Btdagger \cap \At$. Let
$\Bdaggerrk = (\Bdaggerr)^{H_K}$, $\Adaggerrk = (\Adaggerr)^{H_K}$,
$\Btdaggerrk = (\Btdaggerr)^{H_K}$, $\Atdaggerrk =
(\Atdaggerr)^{H_K}$.
\index{Btildedaggern@$\Btdaggern$}
\index{Bdaggern@$\widetilde{\B}^{\dagger, r_n}$}
\index{Btildedagger@$\Btdagger$}
\index{Bdagger@$\Bdagger$}
\index{Atildedaggerr@$\Atdaggerr$}
\index{Adaggerr@$\Adaggerr$}
\index{Adagger@$\Adagger$}
\index{Atildedagger@$\Atdagger$}
\index{Bdaggerrk@$\Bdaggerrk$}
\index{Adaggerrk@$\Adaggerrk$}
\index{Atildedaggerrk@$\Atdaggerrk$}
\index{Btildedaggerrk@$\Btdaggerrk$}


\begin{prop}
  \label{propbdaggergal}
  If $L / K$ be a finite extension then $\Bdagger_L$ is a finite field 
  extension of $\Bdaggerk$ of degree $[L_\infty : K_\infty] = [H_K : H_L]$,
  and if $L / K$ is Galois, then the same holds for $\Bdagger_L / \Bdaggerk$,
  which then has galois group $\textrm{Gal}(L_\infty / K_\infty)$.
\end{prop}
\begin{proof}
  See \cite{cc98}, Proposition II.4.1.
\end{proof}

If $A$ is a ring which is complete for the $p$-adic topology and $X, Y$ are indeterminantes we let
\[
	A\{X, Y\} = \varprojlim_n A[X,Y] / p^n A[X,Y],
\]
that is, $A\{X, Y\}$ is the $p$-adic completion of $A[X,Y]$. Every
element of $A\{X, Y\}$ can be written as $\sum_{i, j \geq 0} a_{ij}
X^i Y^j$ where $a_{ij}$ is a sequence in $A$ tending to $0$ 
in the $p$-adic topology. We let
$r,s \in \mb{N}[1/p] \cup \{ +\infty \}$ such that $r \leq s$. 
By definition one has (in Fr$(\Bt)$) $p /
[\overline{\pi}]^{+\infty} = 1 / [ \overline{\pi}]$ and
$[\overline{\pi}]^{+\infty} / p = 0$. Let
\begin{align*}
	\At_{[r;s]} &= \Atplus \{ \frac{ p }{ [\overline{\pi}]^{r}},
        \frac{[\overline{\pi}]^{s}}{ p} \} \\ &= \Atplus \{X, Y \} /
        ([\overline{\pi}]^{r}X - p, pY - [\overline{\pi}]^{s}, XY -
             [\overline{\pi}]^{s-r}), \\ \Bt_{[r;s]} &=
             \At_{[r;s]}[1/p].
\end{align*}
If $I$ is any interval of $\mb{R} \cup \{ +\infty \}$ we let
$\Bt_I = \bigcap_{[r;s] \subset I} \Bt_{[r;s]}$.
It is clear that if $I \subset J$ are two closed intervals then $\Bt_J
\subset \Bt_I$.  One has a $p$-adic valuation $V_I$ on $\Bt_I$ defined
by the condition $V_I(x) = 0$ if and only if $x \in \At_I \setminus p
\At_I$ and such that the image of $V_I$ is $\mb{Z}$. With this
valuation $\Bt_I$ becomes a $p$-adic Banach space.

The action of $G_F$ on $\Atplus$ extends to $\Atplus [p /
  [\overline{\pi}]^r, [ \overline{\pi} ]^s / p]$ and by continunity
further extends to $\At_I$ and $\Bt_I$.  The Frobenius $\varphi$
extends to a morphism
\[
	\varphi: \Atplus[ \frac{ p }{ [\overline{\pi}]^{r}},
          \frac{[\overline{\pi}]^{s}}{ p}] \longrightarrow \Atplus [
          \frac{ p }{ [\overline{\pi}]^{pr}},
          \frac{[\overline{\pi}]^{ps}}{ p } ]
\]
and finally to a map $\varphi : \At_I \ra \At_{pI}$ for every $I$. 

Berger defines $\Btdaggerrrig = \Bt_{[r, +\infty[},~ \Btdaggerrig =
            \bigcup_{r \geq 0 } \Btdaggerrrig$.
$\Btdaggerrrig$ is endowed with the Fr\'echet topology defined by the
family of valuations $V_I$ for closed subsets $I \subset [r,
  +\infty[$, and subsequently $\Btdaggerrig$ is an LF-space.  One can define 
$\At^{\dagger,r}_{\textrm{rig}}$ as the
    ring of integers of $\Bt^{\dagger,r}_{\textrm{rig}}$ with respect
    to the valuation $V_{[r;r]}$.  We put $\Atdaggerrig = \bigcup_{r
      \geq 0} \At^{\dagger,r}_{\textrm{rig}}$.
One defines $\Bdaggerrigk$ to be the LF-space arising from the completion of the
$\Bdaggerrk$ with
respect to the Fr\'echet topology induced by the $V_I$. 
Further, let $\Bdaggerrig = \Bdaggerrigf \otimes_{\Bdaggerf} \Bdagger$.



\begin{lem}
  \label{lembdaggerrigk}
  \begin{enumerate}
    \item $\Bdaggerrigk = \Bdaggerrigf \otimes_{\Bdaggerf} \Bdaggerk$.
    \item $\Bdaggerrig = \Bdaggerrigk \otimes_{\Bdaggerk} \Bdagger$.
    \item $(\Bdaggerrig)^{H_K} = \Bdaggerrigk$.
  \end{enumerate}
\end{lem}
\begin{proof}
  See \cite{berger02}, section 3.4.
\end{proof}

Berger has shown the existence of  unique map $\log: \Atplus \ra
  \Btdaggerrig[X]$ such that $\log([x]) = \log[x],~ \log(p) = 0$ and
  $\log (xy) = \log (x) + \log(y)$. Hence one defines $\log \pi :=
  \log (\pi)$ and sets  $\Btdaggerlog = \Btdaggerrig[\log \pi]$,
$\Bdaggerlog = \Bdaggerrig[\log
  \pi]$\index{\Bdaggerlogk@$\Bdaggerlogk$} and $\Bdaggerlogk =
\Bdaggerrigk[\log \pi]$\index{Bdaggerlog@$\Bdaggerlog$}. One defines a
monodromy operator $N$ on $\Btdaggerlog$ by extending $N \log \pi := -
p/(p-1)$ in the usual way.

\subsection{$(\varphi, \Gamma_K)$-modules over $\Bdaggerrigk$}



We describe how to extend certain results of \cite{berger02} to 
(in general non-\'etale) $(\varphi, \Gamma_K)$-modules, cf. also 
\cite{berger08b}. 

We make use of the following notation: Suppose
$R$ is a commutative ring equipped with an endomorphism $f: R \ra R$, and 
$M$ is a $R$-module. We may then consider the $R$-module
$R \otimes_{f, R} M$, where $R$ is considered as an $R$-module via 
$r \cdot s := f(r) s$ ($r,s \in R$).

\begin{enumerate}
  \item A \textbf{$(\varphi, \Gamma_K)$-module} $D$ over
    $\Bdaggerrigk$ is a free, finitely generated $\Bdaggerrigk$-module
    with a semi-linear continuous map $\varphi_D$
    (i.e. $\varphi_D(\lambda x) = \varphi(\lambda) \varphi_D(x)$ for
    $\lambda \in \Bdaggerrigk, x \in D$) and a continuous action of
    $\Gamma_K$ which commutes with $\varphi_D$, such that 
    the map
    \[
      \varphi^* : \Bdaggerrigk \otimes_{\varphi, \Bdaggerrigk} D \lra D,~~~
      a \otimes x \longmapsto a \varphi(x)
    \]
    is an isomorphism of $\Bdaggerrigk$-modules.
  \item $(\varphi, \Gamma_K)$-module $D$ over $\Bdaggerrigk$ is
    \textbf{\'etale} (or \textbf{of slope 0}) if there exists $p$-adic 
    representation $V$ such that $D = \Ddaggerrigk(V)$.
\end{enumerate}

For example, for a $p$-adic representation $V$ we set 
$\Ddaggerrigk(V) := (\Bdaggerrigk \otimes_{\mb{Q}_p} V)^{H_K}$. Furthermore,
let us define  $\Ddagger_{\textrm{log}, K}(V) :=  (\Bdaggerlog 
\otimes_{\mb{Q}_p} V)^{H_K}$ and $\D^+_{\textrm{rig}, K} :=  
(\B^+_{\textrm{rig}} \otimes_{\mb{Q}_p} V)^{H_K}$.
Then $\Ddaggerrigk(V)$ is a $(\varphi, \Gamma)$-module over $\Bdaggerrigk$.

Let $D$ be a $(\varphi, \Gamma_K)$-module over $\Bdaggerrigk$. $\varphi_D$ will 
henceforth simply be denoted by $\varphi$.
For the ring $\Bdaggerrigk$ we have a decomposition $\Bdaggerrigk = 
\bigoplus_{i =  0}^{p-1} (1 + \pi)^i \varphi(\Bdaggerrigk)$ so that one may 
define
an operator $\psi$ on $\Bdaggerrigk$ by sending 
$\sum_{i=0}^{p-1} (1+\pi)^i \varphi(x_i)$ to $x_0$, that
extends a similarly defined operator $\psi$ on $\Bdaggerk$. More generally, 
if $D$ is
a $(\varphi, \Gamma_K)$-module over $\Bdaggerrigk$ we have thanks to
condition a) in the definition of $(\varphi, \Gamma)$-modules that there exists 
a unique operator $\psi$ on $D$ that is defined by the same formula and that
and commutes with the action of $\Gamma_K$.

\begin{prop}
  \label{psiexact}
  If $0 \ra D' \ra D \ra D'' \ra 0$ is an exact sequence of 
  $(\varphi, \Gamma_K)$-modules over $\Bdaggerrigk$ then 
  $0 \ra D'^{\psi = 0} \ra D^{\psi = 0} \ra D''^{\psi = 0} \ra 0$
  is an exact sequence of $\Gamma_K$-modules.
\end{prop}
\begin{proof}
  For the proof of the right-exactness one just uses the fact that if
  $x \in D^{\psi = 0}$ then (uniquely) $x = \sum_{i = 1}^{p-1} (1 +
  \pi)^i \varphi(x_i)$ with $x_i \in D$. The compatibility with the
  action of $\Gamma_K$ is clear since it commutes with $\psi$.
\end{proof}

If $L/K$ is a finite extension, we denote the \textbf{restriction} $D|_L$ 
by
\[
  D|_L := \Bdaggerrigl \otimes_{\Bdaggerrigk} D, 
\]
with actions of $\varphi$ and $\Gamma_L$ defined diagonally. Hence, 
$D|_L$ is a $(\varphi, \Gamma_L)$-module over $\Bdaggerrigl$.

The \textbf{dual} $D^*$ of a $(\varphi, \Gamma_K$)-module $D$ over $\Bdaggerrigk$
is defined by
\[
  D^* := \textrm{Hom}_{\Bdaggerrigk} (D, \Bdaggerrigk),
\]
where for $f \in D^*$ the actions of $\Gamma_K$ and $\varphi$ are defined via
\[
  \gamma(f)(x) := \gamma (f(\gamma^{-1} x)),~ \gamma \in \Gamma_K, x \in D,~~~
  \varphi(f)(x) := \sum a_i \varphi(f(x_i)),~ x = \sum a_i \varphi(x_i) \in D.
\]

If $D_1, D_2$ are two $(\varphi, \Gamma_K)$-modules over $\Bdaggerrigk$ then 
the \textbf{tensor product} of $D_1$ and $D_2$ is defined by
\[
  D_1 \otimes D_2 := D_1 \otimes_{\Bdaggerrigk} D_2,
\]
where $\varphi$ and $\Gamma_K$ act diagonally. Note that this does \emph{not}
imply that $\psi$ acts diagonally.

Let $D$ be a $(\varphi, \Gamma_K)$-module over $\Bdaggerrigk$ of rank
$d$. By \cite{berger08b}, Theorem I.3.3 there exists an
$n(D)$\index{nD@$n(D)$} and a unique finite free
$\Bdaggerrndrigk$-module $D^{(n(D))} \subset D$ of rank $d$ with
\begin{enumerate}
  \item $\Bdaggerrigk \otimes_{\Bdaggerrndrigk} D^{(n(D))} = D$,
  \item Let $D^{(n)} = \Bdaggerrnrigk \otimes_{\Bdaggerrndrigk}
    D^{(n(D))} $ for each $n \geq n(D)$. Then $\varphi(D^{(n)})
    \subset D^{(n+1)}$ and the map
    \[
      \Bdaggerrneinsrigk \otimes_{\varphi, \Bdaggerrnrigk} D^{(n)} \ra D^{(n+1)},~~~
      a \otimes x \mapsto a \varphi(x),
    \]
    is an isomorphism.
\end{enumerate}

\subsection{$\Btdaggerrig$-modules and $B$-pairs}

Let us collect some facts about $\varphi$-modules over $\Btdaggerrig$.

\begin{defi}
  Let $h \geq 1$ and $a \in \mb{Z}$. The \textbf{elementary
    $\varphi$-module $M_{a,h}$} is the $\varphi$-module over
  $\Btdaggerrig$ with basis $e_0, \ldots, e_{h-1}$ and $\varphi(e_0) =
  e_1, \ldots, \varphi(e_{h-2}) = e_{h-1}, \varphi(e_{h-1}) = p^a
  e_0$.\index{Mah@$M_{a,h}$}
\end{defi}

\begin{prop}
  If $M$ is a $\varphi$-module over $\Btdaggerrig$ 
  then there exist integers $a_i, h_i$ such that $M \cong
  \bigoplus_i M_{a_i, h_i}$.
\end{prop}
\begin{proof}
    See \cite{kedlaya05}, Theorem 4.5.7.
\end{proof}

\begin{defi}
  Let $M$ be a $\varphi$-module over $\Btdaggerrig$. If $M = M_{a,h}$
  is elementary one defines the \textbf{slope} of $M$ as $\mu(M) =
  a/h$ and one says that $M$ is \textbf{pure} of this slope. In
  general if $M \equiv \bigoplus M_{a_i, h_i}$ one define $\mu(M) = \sum
  \mu(M_{a_i,h_i})$, so that $\mu$ is compatible with short exact
  sequences.
\end{defi}

Let $D$ now be a $(\varphi, \Gamma)$-module over $\Bdaggerrigk$.
One sets $\Be := (\Btdaggerrig[1/t])^{\varphi = 1}$.
\index{Be@$\Be$}
From \cite{berger08a}, Proposition 2.2.6, we know that 
\begin{enumerate}
  \item $\We(D) := (\Btdaggerrig[1/t] \otimes_{\Bdaggerrigk} D)^{\varphi = 1}$
    \index{WeD@$\We(D)$} is a
    free $\Be$-module of rank $d$ which inherits an action of $G_K$,
  \item $\Wplusdr(D) := \Bplusdr \otimes_{\iota_n, \Bdaggerrnrigk}
    D^{(n)}$\index{WplusdrD@$\Wplusdr(D)$} does not depend on $n \gg
    0$ and is a free $\Bplusdr$-module of rank $d$ which inherits an
    action of $G_K$.
\end{enumerate}
With this in mind, Berger defined:
\begin{defi}
  A tuple $W = (W_e, W^+_{\textrm{dR}})$, where $W_e$ is a free
  $\B_e$-module of finite rank equipped with an semi-linear action of
  $G_K$ and $W^+_{\textrm{dR}}$ is a $\Bplusdr$-lattice in $\Bdr
  \otimes_{\B_e} W_e$ that is stable under the action of $G_K$, is
  called a \textbf{B-pair}.
\end{defi}

From \cite{berger08a}, Proposition 2.2.6 it follows that the
  tuple $\W(D) = (\We(D), \Wplusdr(D))$\index{WD@$\W(D)$} actually is
  a $B$-pair.  Furthermore, Berger proved: 

\begin{thm}
  The functor $D \mapsto \W(D)$ gives rise to an equivalence of
  categories between the category of $(\varphi, \Gamma_K)$-modules
  over $\Bdaggerrigk$ and the category of B-pairs.
\end{thm}

One knows (cf. \cite{berger08b}, section 2.2.) how to construct a
functor $\Dt$ from the category of $B$-pairs to the category of
$(\varphi, G_K)$-modules over $\Btdaggerrig$ such that there exists a
unique $(\varphi, \Gamma_K)$-module $\D(W)$ over $\Bdaggerrigk$ with
$\Btdaggerrig \otimes_{\Bdaggerrigk} \D(W) = \Dt(W)$.  Hence, one
has, similarly as in the preceding theorem:
\begin{thm}
  \label{thmbteq}
  The functor $D \mapsto \TD := \Btdaggerrig \otimes_{\Bdaggerrigk} D$
  gives rise to an equivalence of categories between the category of
  $(\varphi, \Gamma_K)$-modules over $\Bdaggerrigk$ and the category
  of $(\varphi, G_K)$-modules over $\Bdaggerrig$.
\end{thm}
We shall also abbreviate $\TDlog = \Btdaggerlog \otimes_{\Bdaggerrigk} D$ and
\index{Dtildedaggerlog@$\TDlog$} $\Wdr(D) := \Bdr 
  \otimes_{\iota_n \Bdaggerrnrigk} D^{(n)}$,\index{WdrD@$\Wdr(D)$}
  which is independent of the choice of $n$ for $n \gg 0$.

It is known that the canonical map
\[
  \Btdaggerrnrig \otimes_{\Be} \We(D) \ra \Btdaggerrnrig
  \otimes_{\Bdaggerrnrigk} D^{(n)},
\]
induced by $a \otimes x \mapsto ax$, is an isomorphism of
$G_K$-modules for every $n \geq n(D)$. One defines the
following map of $G_K$-modules:
\begin{align}
  \label{mapexact}
  \beta: \We(D) \hookrightarrow \Bdr \otimes_{\Be} \We(D) \cong \Bdr
  \otimes_{\iota_n, \Bdaggerrnrigk}(\Btdaggerrnrig \otimes_{\Be}
  \We(D)) \cong \Wdr(D).
\end{align}
We use the same symbol for the map $\beta : \We(D) \ra \Bdr / \Bplusdr 
\otimes_{\Be}
\We(D)$. Set 
$\Wpluse(D) = (\Btdaggerrig \otimes_{\Bdaggerrigk} D)^{\varphi = 1}$.
$\We(D)$.

Let now $W$ be a $B$-pair and set $X^0(W) =
W_e \cap W^+_{\textrm{dR}} \subset W_{\textrm{dR}}$\index{X0@$X^0(W)$} and 
$X^1(W) =
W_{\textrm{dR}} / (W_e + W^+_{\textrm{dR}})$\index{X1@$X^1(W)$}, which are 
nothing but
the kernel and cokernel respectively of the natural map $W_e \ra
W_{\textrm{dR}} / W^+_{\textrm{dR}}$.  Hence, one has
(\cite{berger09}, Theorem 3.1):

\begin{thm}
  \label{thmx0x1}
  If $W$ is a $B$-pair and $\TD = \Dt(W)$, there are natural identifications 
  \begin{enumerate}
    \item $X^0(W) \cong \Wpluse(D)$ and 
      $X^1(W) \cong \TD / (1 - \varphi)$,
    \item $X^0(W) = 0$ if and only if all slopes of $\TD$ are $> 0$;
      $X^1(W) = 0$ if and only if all slopes of $\TD$ are $\leq 0$.
  \end{enumerate}
\end{thm}

We recall the following definition, introduced by Fontaine (see
\cite{fontaine03}):
\begin{defi}
  An \textbf{almost $\mb{C}_p$-representation} is a $p$-adic Banach space
  $X$ equipped with a linear and continuous action of $G_K$ such that 
  there exists a $d \geq 0$ and two (finite-dimensional) $p$-adic
  representations $V_1 \subset X,~ V_2 \subset \mb{C}_p^d$ such that 
  $X / V_1 \cong \mb{C}_p^d / V_2$.
\end{defi}

Berger has shown that $X^0(W)$ and $X^1(W)$ are almost 
$\mb{C}_p$-representations, cf. \cite{berger09}.

\subsection{Cohomology of $(\varphi,\Gamma_K)$-modules}
\label{subsectioncohomology}

Liu (cf. \cite{liu08}) has worked out reasonable definitions for
cohomology of (in general non-\'etale)
$(\varphi, \Gamma_K)$-modules over $\Bk,
\Bdaggerk$ and $\Bdaggerrigk$.

Let $D$ be a $(\varphi, \Gamma_K)$-module over one of these rings and
let $\Delta_K$\index{DeltaK@$\Delta_K$} be a torsion subgroup of
$\Gamma_K$. $\Gamma_K$ is an open subgroup of $\mb{Z}_p^\times$ and
$\Delta_K$ is a finite group of order dividing $p-1$ (or $2$ if $p =
2$). Define the idempotent operator
$p_{\Delta_K}$\index{pDeltaK@$p_{\Delta_K}$} by $p_{\Delta_K} = (1 /
|\Delta_K|) \sum_{\delta \in \Delta_K} \delta$, so that $p_{\Delta_K}$
is the projection from $D$ to $D' := D^{\Delta_K}$\index{D'@$D'$}.
  If $\Gamma'_K := \Gamma_K / \Delta_K$\index{Gamma'K@$\Gamma'_K$} is
procyclic with generator $\gamma_K$ define the exact sequence
\begin{equation}
  C^\bullet_{\varphi, \gamma_K}(D): 
  \xymatrix{
  0 \ar[r] & D'  \ar[r]^<<<<{d_1} & D' \oplus D'  \ar[r]^<<<<<{d_2}  
  & D' \ar[r]		& 0
} \label{herrcomplexphi}
\end{equation}\index{Cbulletphigamma@$C^\bullet_{\varphi, \gamma_K}(D)$}
with
\[
  d_1(x) = ((\varphi - 1)x, (\gamma_K - 1) x),~~~
  d_2(x,y) = (\gamma_K - 1)x - (\varphi - 1)y.
\]
Define for $i \in \mb{Z}$
\[
  H^i(K, D) := H^i(C^\bullet_{\varphi, \gamma_K}(D)),
\]
which is, up to canonical isomorphism, independent of the choice of $\gamma_K$
(cf. \cite{liu08}, section 2), so that we shall now fix 
a choice of $\Delta_K$ and $\gamma_K$.

For applications in Iwasawa-theory one also considers the following complex:
\begin{equation}
  C^\bullet_{\psi, \gamma_K}(D): 
  \xymatrix{
  0 \ar[r] & D'  \ar[r]^<<<<{d_1} & D' \oplus D'  \ar[r]^<<<<<{d_2}  
  & D' \ar[r]		& 0
} \label{herrcomplexpsi}
\end{equation}\index{Cbulletpsigamma@$C^\bullet_{\psi, \gamma_K}(D)$}
with
\[
  d_1(x) = ((\psi - 1)x, (\gamma_K - 1) x),~~~
  d_2(x,y) = (\gamma_K - 1)x - (\psi - 1)y.
\]

If $D_1$ and $D_2$ are two $(\varphi, \Gamma_K)$-modules over
$\Bdaggerrigk$ one may, following Herr (\cite{herr01}), define the
following cup products (we always mean classes where appropriate): 
\begin{equation}
  \begin{array}{llcc}
  &H^0(K, D_1) \times H^0(K, D_2) \lra  H^0(K, D_1 \otimes D_2), ~~~
  (x, y) \mapsto (x \otimes y), \\
  &H^0(K, D_1) \times H^1(K, D_2) \lra  H^1(K, D_1 \otimes D_2), ~~~
  (x, (y, z)) \mapsto (x \otimes y, x \otimes z), \\
  &H^0(K, D_1) \times H^2(K, D_2) \lra  H^2(K, D_1 \otimes D_2), ~~~
  (x, y) \mapsto (x \otimes y), \\
  &H^1(K, D_1) \times H^1(K, D_2) \lra  H^2(K, D_1 \otimes D_2), ~~~
  ((x, y), (w, v)) \mapsto y \otimes \gamma_K(w) - x \otimes \varphi(v).
  \end{array} \!\!\!\!\!\!\!\!\! \label{cupproducts}
\end{equation}
We note that some authors swap the maps of the sequence 
$C^\bullet_{\varphi, \gamma_K}(D)$ so that of course one has to adjust the 
definition of the cup-product. We adhere to the conventions made in 
\cite{herr01}.



Liu's result is then (\cite{liu08}, Theorem 0.1 and Theorem 0.2):
\begin{thm}
  Let $D$ be a $(\varphi, \Gamma_K)$-module over $\Bdaggerrigk$.
  \begin{enumerate}
    \item If $D = \Ddaggerrig(V)$ is \'etale one has canonical
      functorial isomorphisms $H^i(K, \Ddaggerrig(V)) \cong H^i(G_K,
      V)$ for all $i \in \mb{Z}$ that are compatible with cup-products.
    \item $H^i(K, D)$ is a finite dimensional $\mb{Q}_p$-vectorspace 
      and vanishes for $i \not= 0, 1, 2$.
    \item For $i = 0, 1, 2$ the pairing 
      \begin{align*}
        H^i(K, D) \times H^{2-i}(K, D^*(1)) \lra H^2(K, D \otimes D^*(1)) &= 
        H^2(K, \Bdaggerrigk(1)) \\
        &= H^2(K, \mb{Q}_p(1)) \cong \mb{Q}_p
      \end{align*}
      where $D \otimes D^*(1) \ra \Bdaggerrigk(1)$ is the map 
      $x \otimes f \mapsto f(x)$, is perfect.
  \end{enumerate}
\end{thm}

Recall that $D|_L = \Bdaggerrigl \otimes_{\Bdaggerrigk} D$.  Let $m =
[\Delta_K : \Delta_L]$ and $n$ be such that $\gamma_K^{p^n} =
\gamma_L$. Define $\tau_{L/K} = \sum_{i=0}^{p^n-1} \gamma_K^i$ and
$\sigma_{L/K} = \sum_{g \in \Gamma_K / \Gamma_L} g$.
We define the restriction maps $\textrm{Res}: H^i(K, D) \ra H^i(L,
D|_L)$\index{Res@$\textrm{Res}$} via the map induced by the following
map on complexes (where $*'$ means the invariants with respect to the
``right'' $\Delta$):
\[
  \xymatrix{
  0 \ar[r] & D' \ar[d]^{\textrm{id}} \ar[r]^<<<<{d_1} & D' \oplus D' \ar[d]^{\textrm{id} \oplus
    (m \cdot \tau_{L/K}) }
  \ar[r]^<<<<<{d_2}  
  & D' \ar[r]	\ar[d]^{\textrm{id}}	& 0 \\
  0 \ar[r] & D|_L'  \ar[r]^<<<<{d_1} & D|_L' \oplus D|_L'  \ar[r]^<<<<<{d_2}  
  & D|_L' \ar[r]		& 0 
  }
\]
Similarly, we define the corestriction map
$\textrm{Cor}: H^i(K, D) \ra H^i(L, D|_L)$\index{Cor@$\textrm{Cor}$} via the 
map induced by the following map on complexes:
\[
  \xymatrix{
  0 \ar[r] & D|_L' \ar[d]^{\sigma_{L/K}} \ar[r]^<<<<{d_1} & D|_L' \oplus D|_L' 
  \ar[d]^{\sigma_{L/K} \oplus \textrm{id}} \ar[r]^<<<<<{d_2}  
  & D|_L' \ar[r]	\ar[d]^{\textrm{id}}	& 0 \\
  0 \ar[r] & D' \ar[r]^<<<<{d_1} & D' \oplus D'  \ar[r]^<<<<<{d_2}  
  & D' \ar[r]		& 0 
  }
\]

\begin{prop}
  \label{proprescor}
  The map $\textrm{Cor} \circ \textrm{Res}$ on $H^i(K, D)$ is nothing but 
  multiplication by $[L:K]$. 
\end{prop}
\begin{proof}
  It is clear that on $H^0(K, D) = D^{\varphi = 1, \gamma_K = 1}$
  (thus $\gamma_K$ acts trivially) the map $\textrm{Cor} \circ
  \textrm{Res}$ is just the trace map and equal to multiplication by
  $[L:K]$.  Since the $H^i(K, D)$ are cohomological $\delta$-functors
  (see \cite{kedlaya09}, Theorem 8.1)
  we get the claim.
\end{proof}


\subsection{$(\varphi, N, \mbox{Gal}(L/K))$-modules 
associated to $(\varphi, \Gamma_K)$-modules}

We begin with a series of definitions (see \cite{berger02}, section 5, and
\cite{berger08b}).

\begin{defi}
  Let $D$ be $(\varphi, \Gamma_K)$-module and $n \geq n(D)$. Set
  \[
    \Dplusdifn(D) := K_n[[t]] \otimes_{\iota_n, \Bdaggerrnrigk} D^{(n)},~~~
    \Ddifn(D) := K_n((t)) \otimes_{\iota_n, \Bdaggerrnrigk} D^{(n)}
  \]\index{Dplusdifn@$\Dplusdifn(D)$}\index{Ddifn@$\Ddifn(D)$}
  and, via the transition maps $ \Dplusdifn(D) \hookrightarrow \Dplusdifneins,~ f(t) 
  \otimes x \mapsto f(t) \otimes \varphi(x)$ (and similarly for
  $\Ddifn(D) \hookrightarrow \Ddifneins$)
  \[
    \Dplusdif(D) := \varinjlim_n \Dplusdifn(D),~~~
    \Ddif(D) := \varinjlim_n \Ddifn(D).
  \]\index{Dplusdif@$\Dplusdif(D)$}\index{Ddif@$\Ddif(D)$}
\end{defi}

Note that $\Dplusdif(D)$ (resp. $\Ddif(D)$) is a free
$K_\infty[[t]] := \bigcup_{n = 1}^\infty K_n[[t]]$- (resp. 
$K_\infty((t)) = K_\infty[[t]][1/t]$-)module of rank $d$ with 
a semi-linear action of $\Gamma_K$. One defines a $\Gamma_K$-equivariant
injection
\[
  \iota_n : D^{(n)} \hookrightarrow \Dplusdifn(D),~~~ x \mapsto 1 \otimes x.
\]\index{iotan@$\iota_n$}

\begin{defi}
  Let $D$ be a $(\varphi, \Gamma_K)$-module. Set
  \begin{align*}
    \Dcris^K(D) &:= (\Bdaggerrigk[1/t] \otimes_{\Bdaggerrigk} 
    D)^{\Gamma_K}, \\
    \Dst^K(D) &:= (\Bdaggerlogk[1/t] \otimes_{\Bdaggerrigk} D)^{\Gamma_K}, \\
    \Ddr^K(D) &:= (\Ddif(D))^{\Gamma_K},
  \end{align*}\index{DcrisD@$\Dcris(D)$}\index{DstD@$\Dst(D)$}
  \index{DdrD@$\Ddr(D)$}
  and 
  \[
    \mbox{Fil}^i \Ddr^K(D) := \Ddr^K(D) \cap t^i \Dplusdif(D) \subset
    \Ddif(D),~ i \in \mb{Z}.
  \]\index{FiliDdrD@$\textrm{Fil}^i \Ddr(D)$}
  The filtration $\Fil^i \Ddr^K(D)$ is decreasing, separated and
  exhaustive.
  We also set $\Ddr^{K,+}(D) := \Fil^0(\Ddr^K(D)) = \Dplusdif(D)^{\Gamma_K}$.
\end{defi}

One has canonical maps which we will denote by $\alpha_\ast$ for
$\ast \in \{ \textrm{cris}, \textrm{st}, \textrm{dR}\}$, induced by 
$a \otimes d \mapsto ad$:
\begin{align*}
  \Bdaggerrigk[1/t] \otimes \Dcris^K(D) \ra D[1/t] \\
  \Bdaggerlogk[1/t] \otimes \Dst^K(D) \ra \Bdaggerlogk[1/t] \otimes D, \\
  K_\infty((t)) \otimes \Ddr^K(D) \ra \Ddif(D).
\end{align*}

\begin{prop}
  All maps $\alpha_\ast$ above are injective. Hence, one always has inequalities
  \[
  \dim_{K_0} \Dcris^K(D) \leq \dim_{K_0} \Dst^K(D) \leq 
  \dim_K \Ddr^K(D) \leq \mbox{rank}_{\Bdaggerrigk} D,
  \]
  and equalities $\dim \D^K_*(D) = \mbox{rank}_{\Bdaggerrigk} D$ for 
  $* \in \{ \textrm{cris}, \textrm{st}, \textrm{dR} \}$ 
  if and only if the corresponding $\alpha$ is an isomorphism.
\end{prop}
\begin{proof}
  Standard proof.
\end{proof}

\begin{defi}
  The \textbf{Hodge-Tate weights} of a $(\varphi, \Gamma_K)$-module are those 
  integers $h$ such that $\Fil^{-h} \Ddr^K(D) \not= \Fil^{-h+1} \Ddr^K(D)$.
  We say that $D$ is \textbf{positive} if $h \leq 0$ for all weights $h$, and 
  that $D$ is \textbf{negative} if $h \geq 0$ for all weights $h$.
\end{defi}

\begin{prop}
  \label{propdpositive}
  Let $D$ be a de Rham $(\varphi, \Gamma_K)$-module over $\Bdaggerrigk$.
  If $D$ is positive then $\Ddr^{K,+}(D) = \Ddr^{K}(D)$. 
  More generally, let $h \geq 0$ be such that
  $\Fil^{-h} \Ddr^K(D) = \Ddr^K(D)$. Then $t^h \Ddr^K(D) = \Ddr^{K,+}(D(-h))$ 
  (in $\Ddif(D)$).
\end{prop}
\begin{proof}
  The first part is obvious from the definitions and can be shown
  the same way as in the \'etale case. 
  The second follows similarly from Lemma \ref{lemtatetwist}.
\end{proof}

One can define the \textbf{Tate-twist} for a $(\varphi,
\Gamma_K)$-module D:  if $k \in \mb{Z}$, then $D(k)$ is the $(\varphi,
\Gamma_K)$-module with $D$ as $\Bdaggerrigk$-module, but with
\[
  \varphi|_{D(k)} = \varphi|_D,~~~ \gamma x = \chi^k(\gamma) \gamma x,~ x \in D.
\]
Analoguously one define a \textbf{Tate-twist} for a filtered
$(\varphi, N)$-module $D$ over $K_0$.  If $k \in \mb{Z}$, then $D[k]$
is the filtered $(\varphi, N)$-module with $D$ as $K_0$-vectorspace
and filtration $\Fil^r (D[k])_K = \Fil^{r-k} D_K$ and
\[
  N|_{D[k]} = N|_D,~~~ \varphi|_{D[k]} = p^{k} \varphi|_{D}.
\]

\begin{lem}
  \label{lemtatetwist}
  One has $\Dst^K(D(k)) = \Dst^K(D)[-k]$.
\end{lem}
\begin{proof}
  One has $D(k) = D \otimes_{\mb{Z}_p} \mb{Z}_p(k)$, and if $e_k$
  is a generator of $\mb{Z}_p(k)$, the isomorphism
  \[
  (\Bdaggerlogk[1/t] \otimes_{\Bdaggerrigk} D)^{\Gamma_K}[-k] \ra
  (\Bdaggerlogk[1/t] \otimes_{\Bdaggerrigk} D(k))^{\Gamma_K}
  \]
  is given by
  \[
    d = \sum a_n \otimes d_n \mapsto \sum a_n e_{-k} \otimes (d_n
    \otimes e_k) = (e_{-k} \otimes e_k) d.
  \]
\end{proof}

\begin{defi}
  \label{defisemicrys}
  A $(\varphi, \Gamma_K)$-module $D$ is defined to be 
  \textbf{crystalline} (resp. \textbf{semi-stable}, resp. \textbf{de Rham}) if 
  $\dim_{K_0} \Dcris^K(D) = \mbox{rank}_{\Bdaggerrigk} D$ 
  (resp. $\dim_{K_0} \Dst^K(D) = \mbox{rank}_{\Bdaggerrigk} D$, resp. 
   $\dim_K \Ddr^K(D)$ $= \mbox{rank}_{\Bdaggerrigk} D$).

  Similarly, we define $D$ to be \textbf{potentially crystalline}
  (resp.  \textbf{potentially semi-stable}) if there exists a finite
  extension $L/K$ such that $D|_L$ is cristalline (resp. semistable).
\end{defi}

\begin{defi}
  \label{defi_ndr}
  Let $D$ be a de Rham $(\varphi, \Gamma_K)$-module of rank $d$. If $n \geq
  n(D)$, set
  \[
    \Ndr^{(n)}(D) := \{ x \in D^{(n)}[1/t]|~ \iota_m(x) \in K_m[[t]] \otimes_K
    \Ddr^K(D) \mbox{ for any } m \geq n \}
  \]
  and $\Ndr(D) = \varinjlim_n \Ndr^{(n)}(D)$.\index{NdrD@$\Ndr(D)$}
\end{defi}

\begin{defi}
  \begin{enumerate}
  \item For a torsion free element $\gamma_K$ of $\Gamma_K$ 
    Perrin-Riou's differential operator $\nabla$ is defined as
    \[
    \nabla = -\frac{\log(\gamma)}{\log_p (\chi(\gamma_K))} = -
    \frac{1}{\log_p (\chi(\gamma_K))} \sum_{n \geq 1} \frac{(1-
      \gamma_K)^n}{n} \in \mc{H}(\Gamma_K).
    \]\index{nabla@$\nabla$}
  \item The operator $\partial$ (on $\Bdaggerrigk[1/t]$) is defined as
    $\partial := 1/t \cdot \nabla$.\index{partial@$\partial$}
  \end{enumerate}
\end{defi}

We remark that $\nabla$ is independent of the choice of $\gamma$,
  which may be checked with the series properties of $\log$.  The
module $\Ndr(D)$ is denoted by $\D$ in \cite{berger10}, Theorem
III.2.3. This theorem also implies:

\begin{thm}
  Let $D$ be a de Rham $(\varphi, \Gamma_K)$-module of rank $d$. Then
  $\Ndr(D)$ is a $(\varphi, \Gamma_K)$-module of rank $d$ with the 
  following properties:
  \begin{enumerate}
    \item $\Ndr(D)[1/t] = D[1/t]$,
    \item $\nabla_0(\Ndr(D)) \subset t \Ndr(D)$.
  \end{enumerate}
\end{thm}

The following proposition is analoguous to \cite{berger02}, Theorem 3.6.
\begin{prop}
  \label{propgkinvariant}
  Let $D$ be a semistable $(\varphi, \Gamma_K$)-module. Then 
  one has 
  \begin{align*}
    (\Btdaggerrig \otimes_{\Bdaggerrigk} D)^{G_K} &= 
    D^{\Gamma_K},\\
    (\Btdaggerrig[1/t] \otimes_{\Bdaggerrigk} D)^{G_K} &= 
    (\Bdaggerrigk[1/t] \otimes_{\Bdaggerrigk} D)^{\Gamma_K},\\
    (\Btdaggerlog[1/t] \otimes_{\Bdaggerrigk} D)^{G_K} &= 
    (\Bdaggerlogk[1/t] \otimes_{\Bdaggerrigk} D)^{\Gamma_K}.
  \end{align*}
\end{prop}
\begin{proof}
  We only treat the first case, as the proof of the others is similar.

  One has $(\Btdaggerrig \otimes_{\Bdaggerrigk} D)^{G_K} \subset
  (\Btdaggerrig \otimes_{\Bdaggerrigk} D)^{H_K} = \Btdaggerrigk
  \otimes_{\Bdaggerrigk} D$ since $H_K$ acts trivially on $D$ (it is a 
  free $\Bdaggerrigk$-module). 
  Let $\{ e_i \}_{1 \leq i \leq d}$ be a $\Bdaggerrigk$-basis of $D$ and
  $\{d_i\}_{1 \leq i \leq r}$ be a $K_0$-basis for $(\Btdaggerrig
  \otimes_{\Bdaggerrigk} D)^{G_K}$ and $M \in M_{r \times
    d}(\Btdaggerrigk)$ defined by the relation $(d_i) = M(e_i)$. $M$
  has rang $r$ (that is, the image of a basis of $D$ under $M$ form a
  free $\Bdaggerrigk$-module of rank $r$) and satisfies $\gamma_K(M)G
  = M$ (since the elements $d_i$ are fixed under $\gamma_K$), where $G
  \in \mbox{GL}_d(\Bdaggerrigk)$ is the matrix of $\gamma_K$ with
  respect to the basis $\{ e_i \}$.  The operator $R_m$ of
  Colmez/Berger (cf. loc.cit., \S 2.6) give $\gamma_K(R_m(M)) G -
  R_m(M) = 0$ for every $m \in \mb{N}$. Further $R_m(M) \overset{m \to
    \infty}{\lra} M$ and $N = \varphi^m(R_m(M)) \in \Bdaggerrigk$
  since $R_m$ is a section of $\varphi^{-m}(\B^{\dagger, p^k
    r}_{\textrm{rig}, K}) \subset \Bt^{\dagger,r}_{\textrm{rig},
    K}$. Hence, $\gamma(N) \varphi^m(G) = N$ and since the actions of
  $\varphi$ and $\Gamma_K$ commute on $D$ one has $\varphi(G) P =
  \gamma_K(P) G$, where $P \in M_{d}(\Bdaggerrigk)$ is the matrix of
  $\varphi$ with respect to the basis $\{ e_i \}$.  If one sets $Q =
  \varphi^{m-1}(P) \ldots \varphi(P) P$ then $\varphi^m(G) Q =
  \gamma_K(Q) G$ and hence $\gamma_K(N Q) G = NQ$, so that $NQ$
  determines $r$ elements in $D$ that are fixed under $\Gamma_K$. But
  since for $m$ big enough the matrix $M$ has rank $r$ and $P$ has
  full rank, since it is an injection and $\Bdaggerrigk \cdot
  \varphi(D) = D$, one sees that these elements give a rank
  $r$-submodule of $D$. Hence, the $K_0$-vectorspace generated by
  these elements is also of dimension $r$, whence the claim.
\end{proof}

Before stating the next result we recall the notion of a $p$-adic
differential equation.  If $D$ is any $(\varphi, \Gamma_K)$-module
over $\Bdaggerrigk$ it is known that the same definition as for
$\nabla$ gives rise to differential operator $\nabla_D: D \ra D$ that
commutes with the action of $\varphi$ and $\Gamma_K$ such that
$\nabla_D(\lambda x) = \nabla (\lambda) x + \lambda \nabla_D(x)$ (see
\cite{berger08b}, Proposition III.1.1). With this one may also
consider the operator $\partial_D = 1/t \cdot \nabla_D$ on $D[1/t]$.
A \textbf{$p$-adic differential equation} is a $(\varphi, \Gamma_K)$-module $D$
over $\Bdaggerrigk$ that is stable under the operator $\partial_D$.

If there is no confusion we will drop the index $D$ of the operators
$\nabla_D$ and $\partial_D$.

\begin{thm}
  \label{thmandre}
  Let $M$ be a $p$-adic differential equation equipped with a Frobenius. Then
  there exists a finite extension $L/K$ such that the natural map
  \[
    \Bdaggerlogl \otimes_{L'_0} (\Bdaggerlogl \otimes_{\Bdaggerrigk}
    D)^{\partial = 0} \ra \Bdaggerlogl \otimes_{\Bdaggerrigk} D.
  \]
  is an isomorphism.
\end{thm}
\begin{proof}
  \cite{andre02}.
\end{proof}

Recall that a \textbf{$\nabla$-crystal} over $\Bdaggerrigk$ is a
  free $\Bdaggerrigk$-module equipped with an action of a Frobenius
  and a connection (also denoted by $\nabla$), compatible with
  $\nabla$ on $\Bdaggerrigk$, that commutes with the Frobenius.  A
  $\nabla$-crystal over $\Bdaggerrigk$ is called \textbf{unipotent} if
  it admits a filtration of sub-crystals such that each successive
  quotient has a basis consisting of elements in the kernel of
  $\nabla$. More generally, a $\nabla$-crystal $M$ is called
  \textbf{quasi-unipotent} if there exists a finite extension $L/K$
  such that $\Bdaggerrigl \otimes_{\Bdaggerrigk} M$ (which is a 
  $\nabla$-crystal over $\Bdaggerrigl$ in a natural way) is unipotent.

We note the following result, which is known by the
experts and may be proved as in the \'etale case (\cite{berger02},
Proposition 5.6):
\begin{prop}
  Every de Rham $(\varphi, \Gamma_K)$-module is potentially semi-stable.
\end{prop}
\begin{proof}
  One defines the (faithful, exact, ...) functor $D \mapsto \Ndr(D)$
  from the category of de Rham $(\varphi, \Gamma_K)$-modules into 
  the category of $p$-adic differential equations equipped with 
  a Frobenius. Since by Andr\'e's theorem  \ref{thmandre} one knows that 
  any such equation is quasi-unipotent, it suffices to show that
  $D$ is potentially semistable if and only if 
  $\Ndr(D)$ is quasi-unipotent.

  Now $D$ is potentially semistable if and only if there exists a
  finite extension $L/K$ such that
  \[
    \dim_{L_0} (\Bdaggerlogl[1/t] \otimes_{\Bdaggerrigk} D)^{\Gamma_L} = 
    \mbox{rank}_{\Bdaggerrigk} D =: d.
  \]
  This gives via \cite{berger02}, Proposition 5.5 a unipotent
  $\nabla$-subcrystal of $D|_L[1/t]$, which is nothing else but
  $\Ndr(D|_L) \cong \Bdaggerrigl \otimes_{\Bdaggerrigk} \Ndr(D)$.

  Conversely if $D|_{L'}[1/t]$ contains a unipotent
  $\nabla$-subcrystal of rank $d$ for some finite extension $L'/K$
  then the again by loc.cit. there exist elements $e_0, \ldots,
  e_{d-1}$ which generate an $L'_0$-vectorspace of dimension $d$ on
  which $\log(\gamma)$ acts trivially. Hence, there exists a finite
  extension $L/L'$ such that $\Gamma_L$ acts trivially on this basis,
  so that we obtain a basis of $(\Bdaggerlogl[1/t]
  \otimes_{\Bdaggerrigk} D)^{\Gamma_L}$ of the right dimension,
  i.e. $D$ is potentially semistable.  
\end{proof}

We briefly review the slope theory of $\varphi$-modules over
$\Bdaggerrigk$ or $\Bdaggerk$.

\begin{defi}
  Let $M$ is a $\varphi$-module over one of these rings. If $M$ is of
  rank $1$ and $v$ a generator, then $\varphi(v) = \lambda v$ for some
  $\lambda \in (\Bdaggerrigk)^\times = (\Bdaggerk)^\times$
  (cf. \cite{kedlaya05}; see also \cite{kedlaya08}, Hypothesis
  1.4.1. resp.  Example 1.4.2).  We define the \textbf{degree}
  $\deg(M)$ of $M$ to be $w(\lambda)$, where $w$ is the $p$-adic
  valuation of $\Bk$.  If $M$ is of rank $n$ then $\bigwedge^n M$ has
  rank $1$.  We define the \textbf{slope} $\mu(M)$ of $M$ as $\mu(M) =
  \deg(M) / \textrm{rk} M$.
\end{defi}
We remark that the definition of the degree (hence the slope) is
independent of the choice of the generator.  Under the equivalence of
Theorem \ref{thmbteq} we have the following correspondence of the
slope theory: If $D$ is a $(\varphi, \Gamma_K)$-module over
$\Bdaggerrigk$, one may consider the $\varphi$-module $\TD$ over
$\Btdaggerrig$. Then the two definitions of the slope for $D$
coincide. Hence, we have the notion of a $(\varphi, \Gamma_K)$-module
that is \textbf{pure} of some slope. The fundamental theorem is the
following result by Kedlaya:

\begin{thm} (Slope filtration theorem)~
  Let $M$ be a $\varphi$-module over $\Bdaggerrigk$. Then there exists
  a unique filtration $0 = M_0 \subset M_1 \subset \ldots \subset M_l
  = M$ by saturated $\varphi$-submodules whose successive quotients
  are pure with $\mu(M_1 / M_0) < \ldots < \mu(M_{l-1} / M_l)$.  If
  $M$ is a $(\varphi, \Gamma_K)$-module all $M_i$ are $(\varphi,
  \Gamma_K)$-submodules.
\end{thm}
\begin{proof}
  See \cite{kedlaya08}.
\end{proof}
\section{Exponential maps}

\subsection{Bloch-Kato exponential maps for $(\varphi, \Gamma_K)$-modules}

In this section we define short exact sequences associated to
$(\varphi, \Gamma_K)$-modules, generalizing the ``classical''
Bloch-Kato sequence (see \cite{bk93}) which one may use to study
cohomological questions relating to $p$-adic representations (i.e. the
slope zero case).  One interesting phenomenon that occurs in this more
general setting is that, in order to get the general versions of the
exponential maps, it is necessary to distinguish between the slope
$\leq 0$-case and the slope $> 0$-case.

We are interested in the long exact sequences for continuous
Galois-cohomology induced by these sequences. Let us briefly recall
the machinery.  Let $M$ be continuous $G_K$-module and define the
continuous imhomogeneous cochains in the usual way ($q \geq 0$):
\[
  C_{\textrm{cont}}^q(G_K, M) := C_{\textrm{cont}}^q(K, M) := 
  \left\{ x: G^{n} \longrightarrow M|~ x \mbox{ continuous} \right\}
\]\index{CqGK@$C_{\textrm{cont}}^q(G_K, ~)$}
with differential $\delta^q : C_{\textrm{cont}}^q(K, M) \ra C_{\textrm{cont}}^{q+1}(K, M)$ 
defined by
\begin{align*}
  \delta^q(x)(g_1, \ldots, g_{q+1}) = &~ g_1 x(g_2, \ldots, g_{q+1}) + 
  (-1)^{q+1}x(g_1, \ldots, g_q)  \\
  &+\sum_{i = 1}^q (-1)^i x(g_1, \ldots, g_{i-1}, g_i g_{i+1}, g_{i+2}, \ldots, g_{q+1}).
\end{align*}
By convention $C^{-i}(G_K, M) = 0$ for $i > 1$.
The continuous cochain complex is then defined via
\[
  C_{\textrm{cont}}^\bullet(K, M) := \left[ C_{\textrm{cont}}^0(K, M) 
  \overset{\delta^0}{\longrightarrow}  C_{\textrm{cont}}^1(K, M) 
  \overset{\delta^1}{\longrightarrow} \ldots \right],
\]
and one defines continuous cohomology via
\[
  H_{\textrm{cont}}^q(K,M) := H^q(C_{\textrm{cont}}^\bullet(K, M)).
\]\index{HqK@$H_{\textrm{cont}}^q(K, ~)$}

\begin{lem}
  If $0 \ra M' \longrightarrow M \overset{f}{\longrightarrow} M''
  \longrightarrow 0$ is an exact sequence of $G_K$-modules such that
  $f$ admits a continuous (but not necessarily $G_K$-equivariant)
  splitting, then continuous cohomology induces a long exact sequence
  \[
    \ldots \ra H_{\textrm{cont}}^i(K, M') \ra H_{\textrm{cont}}^i(K, M) \ra 
    H_{\textrm{cont}}^i(K, M'') \ra  H_{\textrm{cont}}^{i+1}(K, M') \ra \ldots
  \]
\end{lem}
\begin{proof}
  This is standard, see for example \cite{tate76}, \S 2.
\end{proof}
If there is no possibility of confusion we will drop the subscript
``$\textrm{cont}$''. The splitting property in our setting will
be granted by the following

\begin{prop}
  \label{propsplit}
  If $f: B_1 \lra B_2$ be a linear continuous surjective map of
  $p$-adic Banach spaces, there exists a continuous splitting
  $s: B_2 \lra B_1$ of $f$, i.e. $f \circ s = \textrm{id}_{B_2}$.
\end{prop}
\begin{proof}
  See \cite{colmez98}, Proposition I.1.5, (iii).
\end{proof}

We define the following set $X$, which will be used in the next few statements:
\label{setx}
\[
  X := \{ (x,y,z) \in \TDlog[1/t] \oplus \TDlog[1/t] \oplus 
  \We(D) / \Wplusdr(D) |~ N(y) = (p\varphi - 1)(x) \}.
\]

\begin{lem}
  \label{lembkeins}
  Let $D$ be a $(\varphi, \Gamma)$-module over $\Bdaggerrigk$.  We
  assume $D$ is pure of slope $\mu(D) \leq 0$.  Then one has the
  following exact sequences of $G_K$-modules (cf. (\ref{mapexact}) for
  the definition of $\beta$):
 \begin{align*}
    0 \lra \Wpluse(D) \overset{f}{\lra} \We(D) & \overset{g}{\lra} 
    \Wdr(D) / \Wplusdr(D)\lra 0 \\
     x &\longmapsto \beta(x)
  \end{align*}
  \begin{align*}
    0 \lra \Wpluse(D) \overset{f}{\lra} \TD[1/t]
    &\overset{g}{\lra} \TD[1/t] \oplus
    \Wdr(D) / \Wplusdr(D) \lra 0 \\
    x &\longmapsto ((\varphi - 1)(x), \beta(x))
  \end{align*}
  \begin{align*}
    0 \lra  \Wpluse(D) \overset{f}{\lra} \TDlog[1/t] &\overset{g}{\lra} 
    X \lra 0 \\
    x &\longmapsto (N(x), (\varphi - 1)(x), \beta(x)) \\
  \end{align*}
  Additionally, each $g$ above 
  admits a continuous (not necessarily $G_K$-equivariant) splitting.
\end{lem}
\begin{proof}
  The exactness of the first sequence is tautological, see Theorem
  \ref{thmx0x1}.  For the second recall that for a $\varphi$-module
  $M$ over $\Btdaggerrig$ the map $\varphi - 1 : M[1/t] \ra M[1/t]$ is
  surjective. This implies the exactness of the second sequence. For
  the exactness of the last sequence first observe that the map $g$ is
  well-defined. Recall that $N : \widetilde{D}_{\textrm{log}} \ra
  \widetilde{D}_{\textrm{log}}$ is extended linearly from the operator
  $N$ on $\Btdaggerlog$, so that
  \begin{equation}
    \label{eqN}
    N( \sum_{i \geq 0} d_i \log^i \pi) = - \sum_{i \geq 1} i \cdot d_i
    \log^{i-1} \pi
  \end{equation}
  for $\sum_{i \geq 0} d_i \log^i \pi \in \widetilde{D}_{\log}$.\!
  The exactness at $\TDlog[1/t]$ is clear since  from
    (\ref{eqN}) one has  $(\TDlog[1/t])^{N = 0} = \TD[1/t]$,
  so we only have to check the exactness at $X$.
  The surjectivity of $N: \Btdaggerlog[1/t] \ra \Btdaggerlog[1/t]$,
    which again follows from (\ref{eqN}),
  implies that
  it is enough to check that if $(0, y, z) \in X$ 
  then there exists $x' \in \TD[1/t]$ such that $g(x') = (0,y,z)$, which
  is nothing but exactness of the second sequence.

  The splitting property follows from Proposition \ref{propsplit} for the first
  sequence. For the remaining ones one has to observe that continuous 
  surjections $1 - \varphi$ and $N$ on $\widetilde{D}[1/t]$ have continuous 
  sections, which follows for example from the proof of Proposition 
  2.1.5 of \cite{kedlaya05} for the first map, and is obvious for the monodromy 
operator.
\end{proof}

\begin{lem}
  \label{lembkzwei}
  Let $D$ be a $(\varphi, \Gamma)$-module over $\Bdaggerrigk$.
  We assume $D$ is pure of slope $\mu(D) > 0$.
  Then one has the following exact sequences of $G_K$-modules
  (cf. (\ref{mapexact}) for the definition of $\beta$):
 \begin{align*}
    0 \lra \We(D) &\overset{f}{\lra} 
    \Wdr(D) / \Wplusdr(D) \overset{g}{\lra} \Wdr(D) / (\We(D) + \Wplusdr(D)) \lra 0 \\
     x &\longmapsto \overline{x}
  \end{align*}
 \begin{align*}
    0 \lra \TD[1/t] \overset{f}{\lra} \TD[1/t] \oplus & 
    \Wdr(D) / \Wplusdr(D) \overset{g}{\lra}  \Wdr(D) / (\We(D) + \Wplusdr(D)) 
    \lra 0 \\
    &f: x \longmapsto ((1 - \varphi)(x), \overline{x}) \\
    &g: (x,y) \longmapsto \overline{y}
 \end{align*}
 \begin{align*}
    0 \lra \TDlog[1/t] & \overset{f}{\lra} X
    \overset{g}{\lra} \Wdr(D) / (\We(D) + \Wplusdr(D)) \lra 0 \\
    &f: x \longmapsto (N(x), (\varphi - 1)(x), \overline{x}) \\
    &g: (x,y,z) \longmapsto \overline{z}
 \end{align*}
  Additionally, each $g$ above 
  admits a continuous (not necessarily $G_K$-equivariant) splitting.
\end{lem}
\begin{proof}
  The exactness of the first sequence is again tautological by Theorem
  \ref{thmx0x1}. The rest of the proof follows analoguously to the
  previous proposition.
\end{proof}

Putting everything together, we also see:
\begin{cor}
  \label{corbpairphigamma}
  Let $D$ be a $(\varphi, \Gamma_K)$-module over $\Bdaggerrigk$. Then one
  has the following exact sequence of $G_K$-modules:
  \begin{align*}
    0 \lra X^0(\TD) \overset{i}{\lra} &\TDlog[1/t] \overset{f}{\lra}
    X \overset{p}{\lra} X^1(\TD) \lra 0 \\
    i&: x \longmapsto x \\
    f&: x \longmapsto (N(x), (\varphi - 1)(x), \overline{x}) \\
    p&: (x,y,z) \longmapsto \overline{z}
  \end{align*}
\end{cor}

Following Nakamura, we now define for a $B$-pair $W = (W_e,
W^+_{\textrm{dR}})$ the following complex:
\[
  C^\bullet(G_K, W) := \textrm{cone}(C^\bullet(G_K, W_e) \lra
  C^\bullet(G_K, W_{\textrm{dR}} / W^+_{\textrm{dR}})),
\]
which is induced by the canonical inclusion $W_e \overset{i}{\lra}
W_{\textrm{dR}}$. That is, we have
\[
  C^i(G_K, W) = C^i(G_K, W_e) \oplus C^{i-1}(G_K, W_{\textrm{dR}} /
  W^+_{\textrm{dR}})
\] 
with differentials
\[
  \delta_C^i: C^i(G_K, W) \ni (a, b) \mapsto (\delta_{C^i(G_K, W_e)}^i(a), 
  i(a) - \delta_{C^i(G_K, W_e)}^{i-1}(b))
\]
More generally, one may define the following complexes:
\begin{align*}
  C^\bullet(G_K, W') &:= \textrm{cone}(C^\bullet(G_K, \TD[1/t]) 
  \xymatrix{\ar[r]^{(1-\varphi, i)} & } C^\bullet(G_K, \TD[1/t] \oplus 
  W_{\textrm{dR}} / W^+_{\textrm{dR}})), \\
  C^\bullet(G_K, W'') &:= \textrm{cone}(C^\bullet(G_K, \TDlog[1/t]) 
  \xymatrix{ \ar[r]^{(N, 1-\varphi, i)} & } C^\bullet(G_K, X)),
\end{align*}

We recall:

\begin{lem}
  \label{lemcone}
  Let $0 \ra A \overset{f}{\ra} B \overset{g}{\ra} C \ra 0$ be a short exact
  sequence of continuous $G_K$-modules such that $g$ admits a continuous, but 
  not necessarely $G_K$-equivariant, splitting. 
  We write (by abuse of notation)
  \begin{align*}
    \textrm{cone}(g) &:= \textrm{cone}(C^\bullet(G_K, B) 
    \overset{g_*}{\lra} C^\bullet(G_K, C)) \\
    \textrm{cone}(f) &:= \textrm{cone}(C^\bullet(G_K, A) 
    \overset{f_*}{\lra} C^\bullet(G_K, B)).
  \end{align*}
  \begin{enumerate}
    \item The natural map of complexes
      \[
        \xymatrix{ C^\bullet(G_K, A): \ar[d] & C^0(G_K, A) \ar[d]^{f}
          \ar[r] & C^1(G_K, A) \ar[r] \ar[d]^{(f,0)} & \ldots
          \\ 
          \textrm{cone}(g): & C^0(G_K,
          B) \ar[r] & C^1(G_K, B) \oplus C^0(G_K, C) \ar[r] & \ldots }
      \]
      is a quasi-isomorphism that is compatible with the long exact sequence, i.e.
      the following diagram is commutative:
      \[
        \xymatrix@C=15pt{
          \ldots \ar[r] & H^i(G_K, A) \ar[r] \ar[d] & H^i(G_K, B) \ar[r] \ar@{=}[d] & 
          H^i(G_K, C) \ar@{=}[d] \ar[r]^\delta  & H^{i+1}(G_K, A) \ar[r] \ar[d] & \ldots \\
          \ldots \ar[r] & H^i(\textrm{cone}(g)) \ar[r] & H^i(G_K, B) \ar[r] & H^i(G_K, C) 
          \ar[r]^\delta  & H^{i+1}(\textrm{cone}(g)) \ar[r] & \ldots
        }
      \]
    \item The natural map of complexes
      \[
        \xymatrix{
          C^\bullet(G_K, C)[-1] : & 0 = C^{-1}(G_K, C) \ar[r] & C^0(G_K, C) \ar[r] 
          & \ldots \\
          \textrm{cone}(f): \ar[u]  & C^0(G_K, A) \ar[r] \ar[u]^0 & 
          C^1(G_K, A) \oplus  C^0(G_K, B) \ar[u]^{(0, g)} \ar[r] & \ldots
        }
      \]
      is a quasi-isomorphism that is compatible with the long exact sequence, i.e.
      the following diagram is commutative:
      \[
        \xymatrix@C=15pt{
          \ldots \ar[r] & H^i(G_K, A) \ar[r] \ar@{=}[d] & H^i(G_K, B) \ar[r] \ar@{=}[d] & 
          H^i(G_K, C) \ar[r]^\delta  & H^{i+1}(G_K, A) \ar[r] \ar@{=}[d] & \ldots \\
          \ldots \ar[r] & H^{i}(G_K, A) \ar[r] & H^i(G_K, B) \ar[r]^\delta & 
          H^{i+1}(\textrm{cone}(f)) \ar[u] \ar[r]^\delta  & H^{i+1}(G_K, A) \ar[r] & \ldots
        }
      \]
  \end{enumerate}
\end{lem}
\begin{proof}
  This is left as an exercise, see for example \cite{weibel}, 1.5.8.
\end{proof}

\begin{lem}
  We have canonical quasi-isomorphisms 
  \[
    C^\bullet(G_K, W) \cong C^\bullet(G_K, W') \cong C^\bullet(G_K, W'').
  \]
\end{lem}
\begin{proof}
  Let $W = \W(D)$.  Observe that the inclusions $\We(D) \subset
  \TD[1/t] \subset \TDlog [1/t]$ and $\Wdr(D)$ induce canonical maps
  on these complexes.  If $W = \W(D)$ with $D$ pure of some slope the
  statement then follows from Lemmas \ref{lembkeins}, \ref{lembkzwei}
  and \ref{lemcone}.
  
  For general $D$ we are by Kedlaya's slope filtration theorem reduced
  to the case of an exact sequence $0 \ra D_1 \ra D \ra D_2 \ra 0$
  such that the statement is true for $D_1, D_2$, hence the claim
  follows by considering the long exact sequences associated to this.
\end{proof}

With this statement and the properties of the cone we obtain a 
long exact sequence of cohomology groups:
\[
  \ldots \ra H^i(G_K, W) \ra H^i(G_K, \TDlog[1/t]) \ra H^i(G_K, X) \overset{\delta}{\ra}
  H^{i+1}(G_K, W) \ra \ldots
\]

With these exact sequences in mind we suggest the following
\begin{defi}
  Let $D$ be a $(\varphi, \Gamma_K)$-module over $\Bdaggerrigk$. The transition map 
  \[
    \exp_{K, D}: H^0(K, X) \ra H^1(K, \W(D))
  \]\index{expKD@$\exp_{K,D}$}
  from the exact sequence above is called \textbf{generalized Bloch-Kato 
    exponential map} for $D$. 
\end{defi}

\begin{rem}
  Let $D$ be an \'etale $(\varphi, \Gamma_K)$-module, so that 
  $D = \Ddaggerrigk(V)$ for some $p$-adic representation $V$ of $\Gamma_K$.
  Then since the slope of $D$ is equal to zero, 
  the first exact sequence in Lemma \ref{lembkeins} computes to
  \[
    0 \lra V \ra \Be \otimes_{\mb{Q}_p} V \lra \Bdr / \Bplusdr 
    \otimes_{\mb{Q}_p} V \lra 0
  \]
  This is nothing but the usual Bloch-Kato short exact sequence 
  associated to the $p$-adic representation $V$.
\end{rem}


Recall that if $D$ is any $(\varphi, \Gamma_K)$-module over $\Bdaggerrigk$
the map $\varphi - 1: \TD[1/t] \ra \TD[1/t]$ is surjective. If $x \in \TD$
we write $(\varphi - 1)^{-1}(x)$ for a choice of an element $y \in 
\TD[1/t]$ such that $(\varphi - 1)(y) = x$.
We want to consider the following maps: 
\begin{align*}
  \alpha: \TD &\lra \We(D),~~ x \mapsto \left\{ 
  \begin{array}{cc}
    x,~ \varphi(x) = x,\\
    0,~ \textrm{otherwise}.
  \end{array}
  \right.\\
  \beta: \TD &\lra \Wdr(D) / \Wplusdr(D),~~ x \mapsto 
  \iota_n ((\varphi - 1)^{-1}(x)),
\end{align*}
where the second map is well-defined due to the discussion in
\cite{berger09}, Remark 3.4.  $\alpha$ and $\beta$ are continuous and
fit into the following commutative diagram of $G_K$-modules:
\[
  \xymatrix{
    0 \ar[r] & \TD^{\varphi = 1} \ar[r] \ar@{=}[d] & \TD \ar[r]^{\varphi - 1} 
    \ar[d]^\alpha & 
    \TD \ar[r] \ar[d]^\beta &  \TD / (\varphi - 1) \TD \ar[r] \ar@{=}[d]
    & 0  \\
    0 \ar[r] & X^0(\TD) \ar[r] & \We(D) \ar[r] & \Wdr(D) / \Wplusdr(D) \ar[r]
    & X^1(\TD) \ar[r] & 0,
  }
\]
where we use the identifications for $X^0$ and $X^1$ from Theorem \ref{thmx0x1}.

\begin{prop}
  One has a quasi-isomorphism
  \[
    \textrm{cone}(C^\bullet(G_K, \TD) \overset{\varphi - 1}{\lra}
    C^\bullet(G_K, \TD)) \cong C^\bullet(G_K, \W(D))
  \]
  that is functorial in $D$.
\end{prop}
\begin{proof}
  We denote by $A^\bullet$ the complex on the left hand side of the
  statement.  One checks that the commutativity of the preceeding
  diagram and the cohomological version of \cite{weibel}, Exercise
  1.5.9 show that one has a commutative diagram
  \[\!\!\!\!\!
    \xymatrix@C=10pt{
      \ldots \ar[r] & H^n(G_K, \TD^{\varphi = 1}) \ar[r] \ar@{=}[d] & 
      H^n(A^\bullet) \ar[r] \ar[d] & H^{n-1}(G_K, \frac{\TD}{(\varphi - 1)\TD}) 
      \ar[r] 
      \ar@{=}[d] &  H^{n+1}(G_K, \TD^{\varphi = 1}) \ar[r] \ar@{=}[d] & \ldots \\
      \ldots \ar[r] & H^n(G_K, X^0(\TD)) \ar[r] & H^n(G_K, \W(D)) \ar[r] & 
      H^{n-1}(G_K, X^1(\TD)) \ar[r] & H^{n+1}(G_K, X^0(\TD)) \ar[r] & \ldots
    }
  \]
  which gives the proof.
\end{proof}

Recall the following property of continuous cohomology: If 
$f: M^\bullet \ra N^\bullet$ is map of complexes of continuous $G$-modules 
for some profinite group $G$ one has an identification of complexes
\begin{equation}
  \label{eqcone}
  C^\bullet_{\textrm{cont}}(G, \textrm{cone}(M^\bullet \overset{f}{\ra} 
  N^\bullet)) = \textrm{cone} \left( C^\bullet_{\textrm{cont}}(G, M^\bullet) 
  \overset{f_*}{\ra} C^\bullet_{\textrm{cont}}(G, N^\bullet)
  \right)
\end{equation}
(cf. the discussion in \cite{neko}, 3.4.1.3, 3.4.1.4; it holds in this 
general setting).

We recall that in the derived category of $\Bdaggerrigk$-modules, the
complex $C^\bullet_{\varphi, \gamma}$ is also represented by
\[
  R \Gamma(K, D) = R \Gamma_{\textrm{cont}}(\Gamma_K, \textrm{cone}
  \left[ D \overset{\varphi - 1}{\lra} D \right]) \cong
  \textrm{cone}\left[ R\Gamma_{\textrm{cont}}(\Gamma_K, D) \overset{\varphi - 1}{
      \lra} R\Gamma_{\textrm{cont}}(\Gamma_K, D) \right],
\]
cf. \cite{pottharst12}, section 3.3, where the last identification is
due to (\ref{eqcone}).

The following is then a 
generalization of Proposition \ref{propgkinvariant}:
\begin{prop}
  One has an isomorphism
  \[
    R\Gamma(K, D) \cong R\Gamma(K, \Btdaggerrigk \otimes_{\Bdaggerrigk} D)
  \]
  that is functorial in $D$.
\end{prop}
\begin{proof}
  The proof is similar to \cite{pottharst12}, Proposition 3.8.
  It suffices to show that that the natural map
  \[
    R\Gamma_{\textrm{cont}}(\Gamma_K, D) \lra 
    R\Gamma_{\textrm{cont}}(\Gamma_K, \Btdaggerrigk \otimes_{\Bdaggerrigk} D)
  \]
  is an isomorphism, since applying $\textrm{cone}\left[ \bullet
    \overset{\varphi - 1}{\lra} \bullet\right]$ induces the morphism
  in the statement again due to (\ref{eqcone}).  We apply the
  techniques of \cite{andiov}, Appendix I and use the notation there,
  as follows: Let $\widetilde{\Lambda} := \Btdaggerrrig$, $\mc{G} =
  G_K,~ \mc{H} = H_K$ so that $d = 0$. Further, $\mc{H}' = \mc{H}$,
  $\Lambda^{(i)}_{m, \mc{H}'} = \varphi^{-m}( \B^{\dagger, p^m
    r}_{\textrm{rig}, K})$ (since $i = 0$ is the only possible choice)
  and the maps $\tau^{(i)}_{m, \mc{H}'}$ correspond to the maps $R_m:
  \Bt^{\dagger, r}_{\textrm{rig}, K} \ra \varphi^{-m}( \B^{\dagger,
    p^m r}_{\textrm{rig}, K})$ (cf. \cite{berger02}, Proposition
  2.32).  As in \cite{andiov}, section 7.6, the maps $R_m$ induce maps
  (by the usual process of taking the direct limit over all
  sufficiently big $r$) $R_m : \Btdaggerrigk \otimes_{\Bdaggerrigk} D
  \ra \Btdaggerrigk \otimes_{\Bdaggerrigk} D$ for $m \geq 0$, and as
  in loc.cit.  one obtains a decomposition of $\Gamma_K$-modules
  \[
    \Btdaggerrigk \otimes_{\Bdaggerrigk} D \cong 
    (1 - R_m)(\Btdaggerrigk \otimes_{\Bdaggerrigk} D) \oplus 
    (\Btdaggerrigk \otimes_{\Bdaggerrigk} D)^{R_m = 1}.
  \]
  By construction of the map $R_m$ it is clear that $(\Btdaggerrigk
  \otimes_{\Bdaggerrigk} D)^{R_0 = 1} = D$. Furthermore, as in the
  proof of loc.cit., Proposition 7.7, one may infer that $\gamma_K -
  1$ acts invertibly on $(1 - R_0)(\Btdaggerrigk
  \otimes_{\Bdaggerrigk} D)$, so that $R\Gamma_\textrm{cont}(\Gamma_K,
  (1 - R_0)(\Btdaggerrigk \otimes_{\Bdaggerrigk} D)) = 0$, which gives
  the claim.
\end{proof}

Putting everything together, we see:
\begin{cor}
  \label{corwd}
  One has an isomorphism 
  \[
    R\Gamma(K, D) \cong R\Gamma(G_K, \W(D)).
  \]
  that is functorial in $D$.
\end{cor}
\begin{proof}
  We observe that the natural map 
  \begin{equation}
    \label{eqbtdaggerrigk}
    \Btdaggerrigk \cong R\Gamma_{\textrm{cont}}(H_K, \Btdaggerrig) 
  \end{equation}
  is a quasi-isomorphism. This, together with the preceeding 
  isomorphisms implies
  \begin{align*}
    R\Gamma(K, D) &\cong R\Gamma_{\textrm{cont}}(\Gamma_K, 
    \textrm{cone}(D \overset{\varphi - 1}{\lra} D)) \\
    &= R\Gamma_{\textrm{cont}}(\Gamma_K, 
    \textrm{cone}(\Btdaggerrigk \otimes D \overset{\varphi - 1}{\lra} 
    \Btdaggerrigk \otimes D)) \\
    &= R\Gamma_{\textrm{cont}}(\Gamma_K, R\Gamma_{\textrm{cont}}(H_K, 
    \textrm{cone}(\TD \overset{\varphi - 1}{\lra} \TD))) \\
    &\overset{(*)}{=} R\Gamma_{\textrm{cont}}(G_K, \textrm{cone}(\TD 
    \overset{\varphi - 1}
           {\lra} \TD)) \\
    &= R\Gamma(G_K, \W(D)).
  \end{align*}
  where ($\ast$) holds since the natural map $H^i(G_K / H_K,
  \TD^{H_K}) \ra H^i(G_K, \TD)$ is an isomorphism, since again
  $H^n(H_K, \TD) = 0$ for $n > 0$ due to (\ref{eqbtdaggerrigk}): for
  $i = 1$ this follows from the five term exact sequence in low
  degree, which extends in this case for continuous cohomology
  similarly as in e.g. \cite{nsw}, \S 6, to higher degrees by
  induction.
\end{proof}

\begin{cor}
  \label{remnaka}
  $H^i(G_K, \W(D)) = 0$ for $i \not= 0, 1, 2$ and $H^i(G_K, \W(D))$ is 
  a finite-dimensional $\mb{Q}_p$-vectorspace.
\end{cor}
\begin{proof}
  This follows from the preceeding Corollary and \cite{kedlaya09}, 
  Theorem 8.1.
\end{proof}

We wish to give a more explicit description of the isomorphisms on
cohomology which we will need in the characterizing property of the
big exponential map, where actually only the map for the $H^1$'s will
be important for us. Hence, we may only sketch certain steps for the
higher cohomology groups (that is, $H^2$).

We briefly describe how one may interpret, in the slope $\leq 0$-case,
the cohomology group $H^1(G_K, \Wpluse(D))$ as extensions of
$\mb{Q}_p$ by $\Wpluse(D)$. So let $c \in H^1(G_K, \Wpluse(D))$ and
consider the exact sequence of $G_K$-modules
\[
  0 \lra \Wpluse(D) \lra E_c \lra \mb{Q}_p \lra 0
\]
where $E_c = \mb{Q}_p \oplus \Wpluse(D)$ as $\mb{Q}_p$-vectorspace and
$G_K$ acts on $E_c$ via
\[
  \sigma (a, m) = (a, \sigma m + a c_\sigma). 
\]
Since $c$ is a $1$-cocycle one has
\[
  \sigma (\tau(a,x)) = \sigma (a, \tau x + a c_\tau) = (a, \sigma \tau x +
  a \sigma c_\tau + c_\sigma) = \sigma \tau (a, x),
\]
so that one has a well-defined map $Z^1(K, D) \ra
\textrm{Ext}(\mb{Q}_p, \Wpluse(D))$.  $E_c$ is trivial if and only if
there exists an element $1 \in E_c$ such that $g 1 = 1$ for all $g$,
i.e.
\[
  1 = (1, x), ~~~ g1 - 1 = (0, gx - x + c_g) = 0,
\]
so that $c_g = (1-g)x$ is a coboundary, which implies that the above map 
factors through $B^1(K, D)$. The fact that this map is an isomorphism
can be checked as in the $p$-adic representation case.

\begin{prop}
  \label{propbkeins}
  Suppose we are in the situation of Lemma \ref{lembkeins}.
  Then the complex $C^\bullet_{\varphi, \gamma_K}(K, D)$ (functorially) 
  computes the cohomology of $C^\bullet_{\textrm{cont}}(G_K, X^0(D))$.
\end{prop}
\begin{proof}
  We may assume that $\Gamma_K$ is pro-cyclic with generator $\gamma_K$.
  First we have 
  \[
    H^0(K,D) = D^{\Gamma_K, \varphi = 1} = \TD^{G_K, \varphi = 1} = X^0(D)^{G_K} = 
    H^0(G_K, X^0(D)).
  \]
  thanks to Proposition \ref{propgkinvariant}.
  
  For $H^1$ we apply the construction of Cherbonnier/Colmez
  (\cite{cc99}). To wit, let $(x,y) \in H^1(K, D)$ and pick $b \in
  \TD$ such that $(\varphi - 1) b = x$. Then
  \[
    h^1_{K, D}((x,y)) = \log^0_p(\chi(\gamma)) \cdot
    \left( 
    \sigma \longmapsto \frac{\sigma - 1}{\gamma_K - 1} y  - (\sigma - 1) b
    \right)
  \]
  defines a 1-cocycle with values in $\TD$ but one easily checks that
  $(\varphi - 1) h^1_{K, D}((x,y)) = 0$ so that we actually have a cocycle
  in $H^1(G_K, X^0(D))$. Injectivity and surjectivity now follow in the same way
  as in loc.cit. if one uses the description of extensions of 
  $\mb{Q}_p$ by $X^0(D)$ given above, so that we obtain the isomorphism
  in the $H^1$-case.

  For $H^2$ one can show that since $X^0(D)$ is an almost 
  $\mb{C}_p$-representation that one has a Hochschild-Serre
  spectral sequence $H^i(\Gamma_K, H^j(H_K, X^0(D))) \Rightarrow H^{i+j}(G_K, 
  X^0(D))$ associated to the exact sequence
  $1 \ra H_K \ra G_K \ra \Gamma_K \ra 1$. Since the cohomology on the left 
  vanishes for $j$ or $i$ greater or equal to $2$ one has with 
  the fact that $H^3(G_K, X^0(D)) = 0$ 
 \[
    H^2(G_K, X^0(D)) \cong H^1(\Gamma_K, H^1(H_K, X^0(D))).
  \]
  Now the exact sequence $0 \ra X^0(D) \ra \TD \overset{\varphi - 1}{\ra}
  \TD \ra 0$ of $G_K$-modules gives rise to a sequence
  \[
    \ldots \lra \TD^{H_K} \overset{\varphi - 1}{\lra} \TD^{H_K} \lra
    H^1(H_K, X^0(D)) \lra 0,
  \]
  since $H^1(H_K, \TD) = H^1(H_K, \Btdaggerrig \otimes D) \cong
  H^1(H_K, \Btdaggerrig)^d = 0$. Hence, by Iwasawa theory
  \[
    H^2(G_K, X^0(D)) \cong \TD^{H_K} / (\varphi - 1, \gamma_K - 1).
  \]
    Looking at the quasi-isomorphisms in Corollary \ref{corwd} one sees that 
    using Lemma \ref{lemcone}, 
    since we are in the $X^1(D) = 0$-case, 
    the map $H^2(K,D) \ra H^2(G_K, X^0(D))$ is given by the 
    canonical inclusion of finite-dimensional 
    $\mb{Q}_p$-vectorspaces 
  \[
    H^2(K, D) = D / (\varphi - 1, \gamma_K - 1) \subset 
    \TD^{H_K} / (\varphi - 1, \gamma_K - 1) = H^2(G_K, X^0(D)), 
  \]
  that are of the same dimension.  This gives the description of the
  map for $H^2$.
\end{proof}

\begin{lem}
  \label{leminfres}
  Let $D$ be a $(\varphi, \Gamma_K)$-module over $\Bdaggerrigk$ and
  assume that $\Gamma_K$ is pro-cyclic with generator $\gamma_K$.
  Then one has an exact sequence
  \[
  \begin{array}{ccccccccc}
    0 & \lra & \frac{D^{\varphi = 1}}{(\gamma_K - 1)} & \overset{f}{\lra} 
    & H^1(K, D) & \overset{g}{\lra} & 
    \left( \frac{D}{\varphi - 1} \right)^{\Gamma_K} & \lra & 0 \\
    & & y & \longmapsto & (0, y) \\
    & & & & (x, y) & \longmapsto & x
  \end{array}
  \]
\end{lem}
\begin{proof}
  Recall that by definition
  \[
    H^1(K,D) = \{ (x, y) \in D \oplus D|~ (\gamma_K - 1) x = (\varphi
    - 1) y\} / \{ ((\varphi - 1)z, (\gamma_K - 1) z)|~ z \in D \},
  \]
  so that the first map is well-defined an injective. One checks that the
  map $g$ is well-defined and if $x \in D / (\varphi - 1)$ 
  such that $(\gamma_K - 1) x \in (\varphi - 1)D$ then there
  exists an $y \in D$ such that $(x, y) \in H^1(K, D)$ and 
  $g(x,y) = x$. Obviously $g \circ f = 0$. Let $g(x, y) = 0$ so that 
  $x = (\varphi - 1) z$ for some $z \in D$, so that
  $(x, y) \sim (0, y - (\gamma_K - 1)z)$ in $H^1(K, D)$. Hence, 
  $(x,y)$ is in the image of $f$.
\end{proof}

We remark that this sequence is nothing but the short exact sequence 
associated to the inflation-restriction sequence if $D$ is \'etale, i.e.,
\[
  0 \lra H^1(\Gamma_K, V^{H_K}) \lra H^1(G_K, V) \lra H^1(H_K,
  V^{\Gamma_K}) \lra 0,
\]
see for example \cite{colmez04}, section 5.2.

\begin{prop}
  \label{propbkzwei}
  Suppose we are in the situation of Lemma \ref{lembkzwei}.
  Then the complex $C^\bullet_{\varphi, \gamma_K}(K, D)$ computes the cohomology of
  $C_D^\bullet := C^\bullet_{\textrm{cont}}(G_K, X^1(D))[1]$
\end{prop}
\begin{proof}
  We may assume that $\Gamma_K$ is procyclic with generator $\gamma_K$.
  Since the slope of $D$ is $>0$ one has $X^0(D) = 0$, so that
  $D^{\varphi = 1} = 0$ since $D^{\varphi = 1} \subset \TD^{\varphi =
    1} = 0$, so that $H^0(K, D) = 0$. The same holds tautologically
  for $H^0(C_D^\bullet)$.

  For the case of the $H^1$'s observe that since $X^0(D) = 0$ Lemma
  \ref{leminfres} implies that the canonical map $H^1(K, D) \ra (D /
  (\varphi - 1))^{\Gamma_K}),~ \overline{(x,y)} \mapsto \overline{x},$
  is an isomorphism.  From Theorem \ref{thmx0x1} we also know that
  $X^1(D) = \TD / (\varphi - 1)$. Hence, 
  from Corollary \ref{corwd} and Lemma \ref{lemcone} we have that the map 
  \[
    H^0(G_K, X^1(D)) = \left( \frac{\TD}{\varphi - 1} \right)^{G_K} =
    \left( \frac{D}{\varphi - 1} \right)^{\Gamma_K} \cong ~  H^1(K, D).
  \]
  gives the identification.

  For $H^2$ one has similarly as in the slope $\leq 0$-case a
  Hochschild-Serre spectral sequence $H^i(\Gamma_K, $ $ H^j(H_K,
  X^1(D))) \Rightarrow H^{i+j}(G_K, X^1(D))$. From the exact sequence
  in low degree terms one then has
  \[
    0 \ra H^1(\Gamma_K, H^0(H_K, \TD / (\varphi - 1)) \ra H^1(G_K, \TD
    / (\varphi - 1)) \ra H^0(\Gamma_K, H^1(H_K, \TD / (\varphi - 1)).
  \]
  From the sequence $0 \lra \TD \overset{\varphi - 1}{\lra} \TD \ra
  X^1(D) \lra 0$ one infers the vanishing of $H^1(H_K, $ $X^1(D))$ since
  $H^1(H_K, \TD) = H^2(H_K, TD) = H^2(H_K, \Btdaggerrig)^d = 0$. Hence,
  we see
  \[
    H^1(G_K, X^1(D)) = H^1(\Gamma_K, H^0(H_K, \TD / (\varphi - 1)) = 
    \TD^{H_K} / (\varphi - 1, \gamma_K - 1).
  \]
  so that again by Corollary \ref{corwd} and Lemma \ref{lemcone} the
  canonical inclusion of finite-dimensional $\mb{Q}_p$-vectorspaces
  \[
    H^2(K, D) = D / (\varphi - 1, \gamma_K - 1) \subset 
    \TD^{H_K} / (\varphi - 1, \gamma_K - 1) = H^1(G_K, X^1(D)), 
  \]
  gives the description of the map for $H^2$.
\end{proof}

Finally we describe how one may piece together the isomorphisms
$H^i(K,D) \overset{h^1}{\lra} H^i(K, \W(D))$ in the general case
(where we only make the case $H^1$ explicit, which is all we need for
the application to Perrin-Riou's exponential map): If $(x,y) \in
H^1(K, D)$ write $x = (\varphi - 1)(b') + s(b'')$, 
where $s : \widetilde{D} / (\varphi - 1) \widetilde{D} \ra \widetilde{D}$ is a 
continuous splitting of the natural projection (which exists thanks to 
Proposition \ref{propsplit}), $b' \in \widetilde{D}$ and $b'' \in 
\widetilde{D} / (\varphi - 1) \widetilde{D}$.
Putting the two
constructions together, we may consider the tuple
\begin{align}
  \begin{array}{c}
  h^1(x,y)) := \left(\log^0_p(\chi(\gamma)) \cdot
    \left( 
    \sigma \longmapsto \frac{\sigma - 1}{\gamma_K - 1} y  - (\sigma - 1) b'
    \right), (0,0,\varphi^{-n}((\varphi - 1)^{-1}(s(b'')))
  \right) \\
  ~~~~\in C^1(G_K, \TDlog) \oplus C^0(G_K, X),
  \label{eqheins}
  \end{array}
\end{align} 
and one sees that actually $h^i((x,y)) \in H^1(K, \W(D))$, which gives the 
description of the isomorphism in the general case
by the properties of the mapping cone.

We will briefly describe, similarly as in the slope $\leq 0$-case before,
how one may interpret the cohomology group $H^1(G_K, \We(D))$ as
extensions of $\Be$ by $\We(D)$ (note however that we do not make any 
assumptions about the slopes of $D$). So let $c \in H^1(G_K, \We(D))$
and consider the exact sequence of $G_K$-modules
\[
  0 \lra \We(D) \lra E_c \lra \Be \lra 0,
\]
where $E_c = \Be \oplus \We(D)$ as a $\Be$-module with $G_K$-action 
$\sigma (a, x) = (\sigma a, \sigma x + \sigma a \cdot c_\sigma)$. One has
\[
  \sigma (\tau(a,x)) = \sigma (\tau a, \tau x + \tau a \cdot c_\tau) = 
  (\sigma \tau a, \sigma \tau x +
  \sigma \tau a \cdot \sigma c_\tau + \sigma \tau a \cdot c_\sigma) = \sigma \tau (a, x),
\]
so that one has a well-defined map $Z^1(K, \We(D)) 
\ra \textrm{Ext}(\mb{Q}_p, \Wpluse(D))$. $E_c$ is trivial if and only if 
there exists an element $1 \in E_c$ such that $g 1 = 1$ for all $g$, i.e.
\[
  1 = (1, x), ~~~ g1 - 1 = (0, gx - x + c_g) = 0,
\]
so that $c_g = (1-g)x$ is a coboundary, which implies that the above map 
factors through $B^1(K, \We(D))$. The fact that this map is an isomorphism
can be checked as before.

\begin{prop}
  \label{cohomwe}
  Let $D$ be a $(\varphi, \Gamma_K)$-module over $\Bdaggerrigk$.  Then
  the complex $C^\bullet_{\varphi, \gamma_K}(K, $ $D[1/t])$ computes
  the cohomology of $C^\bullet_{\textrm{cont}}(G_K, \We(D))$.
\end{prop}
\begin{proof}
  The proof is similar to the ones before; in fact, one may reduce to
  the case of Corollary \ref{corbpairphigamma}  by taking direct limits
  (see also \cite{naka12}, Theorem 4.5). We are interested in the
  explicit description of the maps.
  From Proposition \ref{propgkinvariant} again we have:
  \[
    H^0(K, D[1/t]) = D[1/t]^{\varphi = 1, \Gamma_K} = 
    \TD[1/t]^{\varphi = 1, G_K} = H^0(G_K, \We(D)).
  \]
  For $H^1$ we apply the same construction as in Proposition
  \ref{propbkeins}. So let $(x,y) \in H^1(K, D[1/t])$ and pick $b \in
  \TD[1/t]$ such that $(\varphi - 1) b = x$. Then
  \[
    h^1_{K, D}((x,y)) = \log^0_p(\chi(\gamma)) \cdot
    \left( 
    \sigma \longmapsto \frac{\sigma - 1}{\gamma_K - 1} y  - (\sigma - 1) b
    \right)
  \]
  defines a 1-cocycle with values in $\TD[1/t]$ which lies actually in $\We(D)$.
  Injectivity and surjectivity now follow in the same way
  as in loc.cit. if one uses the description of extensions of 
  $\Be$ by $\We(D)$ given above, so that we obtain the isomorphism
  in the $H^1$-case.

  The case of the $H^2$'s follows in the same way as in Proposition 
  \ref{propbkeins}.
\end{proof}

\begin{prop}
  \label{propwdr}
  One has an identification $H^0(K, \Wdr(D)) = \Ddr^K(D)$
\end{prop}
\begin{proof}
    From \cite{cc99}, Proposition IV.1.1 (i) we know that
    $K_\infty[[t]]$ is dense in $(\Bplusdr)^{H_K}$, and the inclusion 
    is compatible the action of $\Gamma_K$. Also one has
    $(\Bplusdr)^{G_K} = K_\infty[[t]]^{\Gamma_K} = K$. Since $D$ is free as a
    $\Bdaggerrigk$-module with trivial $H_K$-action, we see that
    $(\Bplusdr \otimes D)^{G_K} = ((\Bplusdr)^{H_K} \otimes D)^{\Gamma_K} =
    \Dplusdif(D)^{\Gamma_K}$. Since 
    $\Bdr = \varinjlim_{n \geq 0} 1/t^n \cdot \Bplusdr$ and 
    $K_\infty((t)) = \varinjlim_{n \geq 0} 1/t^n \cdot K_\infty[[t]]$ the
    claim follows, since taking invariants is compatible with direct limits.

    Alternatively, the claim also follows from \cite{fontaine03}, Theorem
    2.14, B) i).
\end{proof}

We shall make use of the following considerations. Let $D$ be a
semi-stable $(\varphi, \Gamma_K)$-module over $\Bdaggerrigk$ and
consider the following complex $\mf{C}_{\textrm{st}}(K, D)$
(concentrated in degrees 0, 1, 2):
\begin{equation}
  \begin{array}{cccccc}
     & \Dst^K(D) & \ra & \Dst^K(D) \oplus
    \Dst^K(D) \oplus \Ddr^K(D) / \textrm{Fil}^0 \Ddr^K(D) & \ra &
    \Dst^K(D) \\ & x & \mapsto & (N(x), (\varphi - 1)(x), \beta(x))
    \\ & & & (x,y,z) & \mapsto & N(x)- (p \varphi - 1)(y).
  \end{array} \label{eqh1mfc}
\end{equation}\index{CstKD@$\mf{C}_{\textrm{st}}(K,D)$}
Then an element in $H^1(\mf{C}_{\textrm{st}}(K, D))$ can be considered
as an element in $H^0(K, X)$ and hence be mapped via the exponential
map to $H^1(K, \W(D))$.

We shall give two maps which will be important in the construction of the
dual exponential map for de Rham $(\varphi, \Gamma_K)$-modules. 

First we remark that the canonical inclusion $D \ra \Wdr(D)$ factors
via $D \ra D[1/t]$.  This allows us to describe a map $H^1(K, D) \ra
H^1(G_K, \Wdr(D))$ explicitly via the composition of the canonical map
$H^1(K, D) \ra H^1(K, D[1/t])$, the identification $H^1(K, D[1/t])
\overset{\sim}{\ra} H^1(G_K, \We(D))$ (cf. Proposition \ref{cohomwe})
and the canonical map $H^1(K, \We(D)) \lra H^1(K, $ $\Wdr(D))$.

Secondly, we show that the map
\begin{equation}
  \label{katodual}
  \Ddr^K(D) \lra H^1(G_K, \Wdr(D)),~~~ x \longmapsto [ g \mapsto \log 
  (\chi(\overline{g})) x ] 
\end{equation}
which generalizes Kato's formula of \cite{kato93}, \S II.1, 
is an isomorphism, which may be proved as follows.  First observe that
\[
H^1(G_K, \Bdr \otimes D) \cong H^1(G_K, \Bdr \otimes_{K} \Ddr^K(D)) =
H^1(G_K, \Bdr) \otimes_K \Ddr^K(D).
\]
From \cite{fontouy}, Proposition
5.25, one knows that $K = H^0(G_K, \Bdr) \ra H^1(G_K, \Bdr),$ $x
\mapsto x \cdot \log \chi$ is an isomorphism. This gives the claim.

\begin{defi}
  The \textbf{generalized Bloch-Kato dual exponential map}
  $\exp^*_{K, D^*(1)}$\index{expstarKD@$\exp^*_{K, D^*(1))}$} 
  is the composition of the above maps $H^1(K, D) \ra
  H^1(G_K, \Wdr(D))$ with the inverse of the isomorphism $ \Ddr^K(D)
  \overset{\sim}{\ra} H^1(G_K, \Wdr(D))$.
\end{defi}

Of course, in the \'etale case this is nothing but the dual exponential map
considered by Kato in \cite{kato93}. But even in this more general case
this map has the desired property with respect to adjunction via pairings.
First recall that one may define the $K$-bilinear perfect pairing
$[~,~]_{K, D}$ by the natural map
\[
  [~,~]_{K, D}: \Ddr^K(D) \times \Ddr^K(D^*(1)) \overset{\textrm{ev}}{\lra}
  \Ddr^K(\Bdaggerrigk(1)) \lra K.
\]\index{.pairingddr@$[~,~]_{K,D}$}

For the next proposition we note that Nakamura uses a different
definition of the dual exponential map (see \cite{naka12}, section
2.4), 
  which we briefly recall (we refer to loc.cit for the proofs): one may 
  define the cohomology
  groups $H^i(K, \Ddif(D))$ by $H^i_{\textrm{cont}}(\Gamma_K,
  \Ddif(D))$, which is computed by the complex
  \[
    C^\bullet_{\gamma, \Delta}(\Ddif(D)): ~
    \Ddif(D) \overset{\gamma - 1}{\lra} \Ddif(D).
  \]
  Since the natural map $K_\infty((t)) \otimes_{K} \Ddr^K(D) \ra
  \Ddif(D)$ is an isomorphism one has an identification
  \[
    g_D: \Ddr^K(D) \overset{\sim}{\lra} H^1(K, \Ddif(D)),~~
    x \mapsto \overline{(\log \chi(\gamma))1 \otimes x}.
  \]
  The second definition of $\exp^*_{K, D}$ is then given by the
  composition of the map $H^1(K,D) \ra H^1(K, \Ddif(D)),~ [(x,y)]
  \mapsto \iota_n(y)$ (for $n$ big enough) and the inverse of $g_D$.
  Since $H^i(H_K, $ $\Bdr) = 0$ for $i > 0$ the five term exact
  sequence gives $H^1(G_K, \Wdr(D)) \cong H^1(\Gamma_K, \Bdr^{H_K}
  \otimes D)$. Using the same argument as in Proposition \ref{propwdr}
  one sees that the natural map $H^1(K, \Ddif(D)) \ra H^1(G_K,
  \Wdr(D))$ is an isomorphism. Further, the natural map $H^1(K,D) \ra
  H^1(G_K, $ $\Wdr(D))$ defined before is also given by $[(x,y)] \mapsto
  \iota_n(y)$.  Hence, using all these identifications one obtains
  a commutative diagram
  \[
  \xymatrix@C=30pt{
    H^1(K,D) \ar@{=}[d] \ar[r] & H^1(K, \Ddif(D)) 
    \ar[d]^{\sim} & \Ddr^K(D) \ar[l]_-{\sim} \ar[d]^{\sim} \\
    H^1(K,D) \ar[r] & H^1(G_K, \Wdr(D)) & \ar[l]_-{\sim} H^0(G_K, \Wdr(D)),
  }
  \]
  which shows that the two definitions of $\exp^*$ coincide.

\begin{prop}
  \label{propadjoint}
  Let $D$ be a de Rham $(\varphi, \Gamma_K)$-module over
  $\Bdaggerrigk$ and let $x \in \Ddr^K(D)$ and $y \in H^1(K,
  D^*(1))$. Then
  \[
    \langle \exp_{K, D}(x), y \rangle_{K, D} = \textrm{Tr}_{K / \mb{Q}_p}
    [ x, \exp^*_{K, D}(y) ]_{K, D}
  \]
\end{prop}
\begin{proof}
  See \cite{naka12}, Proposition 2.16.
\end{proof}

\begin{prop}
  \label{propdualexp}
  Let $D$ be a semi-stable $(\varphi, \Gamma_K)$-module over $\Bdaggerrigk$.
  Let $y \in D^{\psi = 1}$ and consider $y$ as $y \in
  (\Bdaggerlogk[1/t] \otimes_F \Dst^K(D))^{N = 0, \psi = 1}$ via the
  comparison isomorphism. Then for $n \gg 0$
  \[
    \exp^*_{V^*(1)}(h^1_{D, K_n}(y)) = p^{-n} \varphi^{-n}(y)(0).
  \]	
\end{prop}
\begin{proof}
  As before we have 
  \[
    h^1_{D, K_n}(y)(\sigma) = \frac{\sigma - 1}{\gamma_{K_n} - 1} y - 
    (\sigma - 1) b,
  \]
  with $(\gamma_{K_n} - 1) (\varphi - 1)b = (\varphi - 1)y$ for some
  $b \in \TD[1/t]$. 
  Further Let
  $n$ be big enough so that we may embed this cocycle into $\Bdr
  \otimes D$, hence $\varphi^{-n}(y) \in K_n((t)) \otimes \Dst^K(D)$
  and
  we may consider $\varphi^{-n}(b)$ as an element in $\Bdr \otimes
  D$. Since $\gamma_{K_n} t = \chi(\gamma_{K_n}) t$ the action of $\gamma_{K_n} -
  1$ is invertible on $t^k K_n \otimes \Dst^K(D)$ for every $k \not=
  0$. Putting this together we see that $h^1_{ D, K_n}$ is equivalent
  in $H^1(K_n, \Bdr \otimes D)$ to
  \[
    \sigma \longmapsto \frac{\sigma - 1}{\gamma_{K_n} - 1}
    (\varphi^{-n}(y))(0).
  \]
    $\sigma$ acts via its image $\overline{\sigma} \in \Gamma^n_K$ (trivially)
    on $K_n$. Furthermore, if $n_i \in \mb{Z}$ is
  a sequence such that $\overline{\sigma} = \lim_{i \to \infty} \gamma_{K_n}^{n_i}$
  one checks by going to the limit that 
  \[
    \frac{\sigma - 1}{\gamma_{K_n} - 1} \frac{\log_p \chi(\gamma_{K_n})}
         {\log_p \chi(\overline{\sigma})}
  \]
  acts trivially on $K_n$. Hence, the above cycle is equivalent to 
  \[
  \sigma \longmapsto p^{-n} \log (\chi(\overline{\sigma})) 
  (\varphi^{-n}(y))(0)
  \]
  The claim follows now from formula (\ref{katodual}).
\end{proof}

\subsection{Perrin-Riou exponential maps for $(\varphi, \Gamma_K)$-modules}

We make the following definitions:
\begin{defi}
  Let $M$ be a $(\varphi, N)$-module over $F$.
  Define $\Ndr(M) = (\Bdaggerlogk \otimes_F M)^{N = 0}$, where
  $N = 1 \otimes N + N \otimes 1$ on $\Bdaggerlogk \otimes_F M$.
  \index{NdrM@$\Ndr(M)$}
\end{defi}
If $D$ is a semi-stable $(\varphi, \Gamma_K)$-module over
$\Bdaggerrigk$ then $\Ndr(\Dst^K(D)) = \Ndr(D)$ 
(see Definition \ref{defi_ndr}).
\begin{defi}
  Let $D$ be a de Rham $(\varphi, \Gamma_K)$-module over $\Bdaggerrigk$.
  \begin{enumerate}
    \item Let $\Dinftyg(D)$\index{Dinftyg@$\Dinftyg(D)$} be the
      submodule of elements $g \in \Ndr(D)^{\psi = 0}$ such that there
      exists an $r \in \mb{Z}$ such that the equation $(1 - p^r
      \varphi) G = \partial^r(g)$ has a solution in $G \in
      \Ndr(D)^{\psi = p^r}$.
    \item Let $\Dinftyf(D)$\index{Dinftyf@$\Dinftyf(D)$} be the
      submodule of elements $g \in \Ndr(D)^{\psi = 0}$ such that there
      exists a family $(G_k)_{k \in \mb{Z}}$ of elements $G_k \in
      \Ndr(D)$ with $\partial (G_k) = G_{k+1}$ and an $r \in \mb{Z}$
      such that $(1 - p^r \varphi) G = \partial^r(g)$
    \item Let $\Dinftye(D)$\index{Dinftye@$\Dinftye(D)$} be the
      submodule of elements $g \in \Ndr(D)^{\psi = 0}$ such that the
      equation $(1 - p^r \varphi) G = \partial^r(g)$ has a solution in
      $G \in \Ndr(D)^{\psi = p^r}$ for every $r \in \mb{Z}$.
  \end{enumerate}
\end{defi}
We first note that if $D \lra D'$ is a morphism of
two de Rham $(\varphi, \Gamma_K)$-modules over $\Bdaggerrigk$ then
this induces a map of $\Gamma_K$-modules $\Dinftyast(D) \ra
\Dinftyast(D')$. Also, one clearly has
\[
  \Dinftye(D) \subset \Dinftyf(D) \subset \Dinftyg(D) \subset 
  \Ndr(D)^{\psi = 0}.
\]
By the above definition one may also define the modules
$\Dinftyast(~)$ by starting with a $(\varphi, \N)$-module.

\begin{defi}
  Let $D$ be a de Rham $(\varphi, \Gamma_K)$-module over
  $\Bdaggerrigk$. We say that $D$ is of \textbf{Perrin-Riou-type} (or
  of PR-type) if $D$ is semistable and $K_0 = K'_0$.
\end{defi}


\begin{lem}
  \label{deltasurj}
  The map $\partial: \Bdaggerlogk \ra \Bdaggerlogk$ is surjective.
\end{lem}
\begin{proof}
  This amounts to an integration of power-series, cf. \cite{berger02},
  Proposition 4.4.
\end{proof}

\begin{lem}
  \label{deltakernel}
  Suppose $K_0 = K'_0$. Then the kernel of $\partial$ on $\Bdaggerlogk$ is equal
  to $K_0$.
\end{lem}
\begin{proof}
  Let $f \in \Bdaggerrigk$. Due to Proposition \ref{propbdaggergal} and
  Lemma \ref{lembdaggerrigk} there is a polynomial $P$ in $\Bdaggerrigf$
  such that $P(f) = 0$ and $P'(f) \not= 0$. Then $\partial(f) = -
  (\partial P)(f) / P'(f)$, so that $\partial(f) = 0$ if and only if
  $f \in K_0$.

  Now suppose $f = \sum_{i=1}^r f_i \log^i \pi \in \Bdaggerlogk$ and
  $\partial(f) = 0$.  Since $\log \pi$ is a transcendent element over
  any $\Bdaggerrigk$ this gives rise to relations $\partial(f_i) +
  (j+1) \frac{\pi+1}{\pi} f_{i+1} = 0$ with $f_{r+1} = 0$. For $i = r$
  this implies $f_r = \lambda \in K_0$, hence $\partial(f_{r-1}) =
  -\lambda r \frac{\pi + 1}{\pi}$. Suppose there exists an $f \in
  \Bdaggerrigk$ with $\partial (f) = \frac{1 + \pi}{\pi}$.  Then
  $\partial (\log \pi - f) = 0$, so that $\log \pi = f + a$ with $a
  \in K_0$, a contradiction to the transcendency property of $\log
  \pi$.  Hence, $\frac{\pi + 1}{\pi}$ is not an element in the image of
  $\partial$ on $\Bdaggerrigk$, and we obtain $\lambda = 0$. By
  recurrence this shows that the kernel of $\partial$ on
  $\Bdaggerlogk$ is contained in $K_0$.
\end{proof}

Let again $D$ be a de Rham $(\varphi, \Gamma)$-module over $\Bdaggerrigk$.

\begin{lem}
  \label{deltasurjd}
  Let $D$ be of PR-type.
  Then the map $\partial: \Bdaggerlogk \otimes \Ndr(D) \ra \Bdaggerlogk \otimes \Ndr(D)$
  is surjective.
\end{lem}
\begin{proof}
  We have 
  \[
    \Bdaggerlogk \otimes_{\Bdaggerrigk} \Ndr(D) = \Bdaggerlogk \otimes_{K_0} \Dst^K(D),
  \]
  whence the claim follows from the Lemma above.
\end{proof}

\begin{prop}
  Let $D$ be of PR-type. The map 
  \[
    \partial: \Ndr(D)^{\psi = 0} \lra \Ndr(D[1])^{\psi = 0}(1)
  \]
  is an isomorphism of $\Gamma_K$-modules.
\end{prop}
\begin{proof}
  With our preparations,
  namely, Lemma \ref{deltasurj} and Lemma \ref{deltakernel},
  this proof works the same as in \cite{perrin01}, 
  Proposition 2.2.3.
\end{proof}

Obviously the operator $\partial$ induces a map of $\Gamma_K$-modules
\[
  \partial: \Ndr(D)^{\psi = 1} \ra \Ndr(D[1])^{\psi = 1}(1)
\]
which
however is in general neither injective nor surjective.  This should
be contrasted with the \'etale case where $D^{\psi = 1} =
\Ddaggerrig(V)^{\psi = 1} = H^1(K, V \otimes_{\mb{Q}_p} \mc{H}
(\Gamma_K))$ and the fact that $\partial$ in this setting
corresponds to the Tate-twist isomorphism.

For a semistable $(\varphi, \Gamma_K)$-module consider the following complex:
\[
  \mf{C}_K(D):~~  0 \ra \Dst^K(D) \overset{\delta_0}{\longrightarrow} \Dst^K(D) 
  \times \Dst^K(D) \overset{\delta_1}{\longrightarrow} \Dst^K(D) \ra 0
\]\index{CKD@$\mf{C}_K(D)$}
with
\begin{align*}
  \delta_0(\nu) &= (N \nu, (1 - \varphi) \nu), \\
  \delta_1(\lambda, \mu) &= N \mu - (1 - p \varphi) \lambda.
\end{align*}
Hence,
\begin{align*}
  H^0(\mf{C}_K(D)) &= \Dst^K(D)^{\varphi = 1, N = 0}, \\
  H^1(\mf{C}_K(D)) &= \{ (\lambda, \mu) \in \Dst^K(D) \times \Dst^K(D)|~ 
  N \mu = (1 - p\varphi) \lambda \} / \delta_0(\Dst^K(D)), \\
  H^2(\mf{C}_K(D)) &= \Dst^K(D) / (N, 1 - p\varphi) \Dst^K(D).
\end{align*}
One also checks that 
\begin{equation}
  \begin{array}{ccccccccc}
    0 & \lra & \frac{\Dst^K(D)^{N = 0}}{ (\varphi - 1) \Dst^K(D)^{N = 0}}
    & \lra & H^1(\mf{C}_K(D)) & \lra & 
    \frac{\Dst^K(D)}{N \Dst^K(D))^{\varphi = p^{-1}}} &
    \lra  & 0 \\
    & & \mu  & \longmapsto & (0, \mu)  \\
    & & & & (\lambda, \mu) & \longmapsto & \lambda
  \end{array} \label{eqh1ck}
\end{equation}
furnishes an exact sequence for $H^1(\mf{C}(D))$.

We see that 
$H^0(\mf{C}(D(k))) = 0$ for $k \gg 0$ resp. $k \ll 0$
since the groups $\Dst(D(k))^{\varphi = 1}$ and 
$(\varphi - 1) \Dst(D(k))$ vanish for those $k$. Similarly, 
$H^1(\mf{C}(D(k))) = 0$ for $k \gg 0$ resp. $k \ll 0$.

Now let $D$ be a de Rham $(\varphi, \Gamma_K)$-module and fix a finite
extension $L/K$ such that $D|_L$ is semistable with $L_0 = L_0'$.

\begin{lem}
  \label{lemndrpsi1}
  Let $k \in \mb{N}$.
%
  Then one has an exact sequence of $\Gamma_K$-modules
  \begin{align*}
    0 \ra \bigoplus_{-k \leq i < 0} H^0(\mf{C}(D|_L(-i)))(i) \cap 
    \Ndr(D(k))^{\psi = 1}(-k))  \ra \Ndr(D(k))^{\psi = 1}(-k) \\
    \overset{\partial^k}{\longrightarrow} \Ndr(D)^{\psi = 1} 
    \overset{\widetilde{\mc{R}}_D}{\ra}
    \bigoplus_{-k \leq i < 0} H^1(\mf{C}(D|_L(-i)))(i) 
  \end{align*}
\end{lem}
\begin{proof}
  The proof may be done in an analogous way as in \cite{perrin01},
  Lemma 2.2.5.  We give a description of the map
  $\widetilde{\mc{R}}_D$ following the definition of a map $\mc{R}_D$
  (cf. equation (\ref{mcrd})) since the constructions which give rise
  to it will be important later on. We just briefly mention
  that this map depends on the inclusion 
  $\Ndr(D) \subset \Ndr(D|_L)$ which is induced by the inclusion 
  $D \subset D|_L$.
\end{proof}

From the lemma we see that, by considering the possible eigenvalues for
$\varphi$,
\begin{align}
  \label{dinftye}
  \Dinftye(D) &= \partial^h ( 1 - p^{-h} \varphi) \Ndr(D(h))^{\psi =
    1},\\ \Dinftyg(D) &= \partial^{-h} ( 1 - p^{h} \varphi)
  \Ndr(D(-h))^{\psi = 1}
  \label{dinftyg}
\end{align}
for $h \gg 0$ since the $H^i(\mf{C}(D))$, $i = 0, 1$, vanish in this
case.  More precisely, for \'etale $(\varphi, \Gamma_K)$-module one
has the following:
\begin{lem}
  \label{lemhetale}
  Let $D = \Ddaggerrig(V)$ for a $p$-adic representation $V$ that is
  de Rham.  Let $h \geq 1$ be such that $\Fil^{-h} \Ddr^K(D) =
  \Ddr^K(D)$.  Then $\Dinftyg(D) = \partial^{-(h+1)} ( 1 - p^{h+1}
  \varphi) \Ndr(D(-(h+1)))^{\psi = 1}$.
\end{lem}
\begin{proof}
  We may reduce to the case that $D$ is semi-stable with $K_0 = K_0'$
  and further by twisting that $h=1$.  We have to check that
  $\partial: \Ndr(D(-2))^{\psi = 1} \ra \Ndr(D(-3))^{\psi = 1}(1)$ is
  an isomorphism, i.e., we have to check the vanishing of
  $H^0(\mf{C}(D(-2)))$ and $H^1(\mf{C}(D(-2)))$. For the first this is
  obvious since for an admissible filtered $(\varphi, N)$-module
  that is positive the eigenvalues of the Frobenius are
  positive. Similarly, thanks to the exact sequence (\ref{eqh1ck}), we
  see that the $H^1$-part vanishes.
\end{proof}

\begin{rem}
  We suspect that in the cases where $V$ is as above and does not
  contain the subrepresentation $\mb{Q}_p(h)$ one actually has
  $\Dinftyg(D) = \partial^{-h} ( 1 - p^{h} \varphi) \Ndr(D(-h))^{\psi
    = 1}$.  This would fit in with the characterizing description of
  the big exponential map in the \'etale case; cf. also the discussion
  in \cite{perrin99}, section 5.1.
\end{rem}

We recall the application $\mc{R}_D$\index{RD@$\mc{R}_D$}. For our
purposes (since we may restrict/corestrict) it will be enough for this
part to assume that $D$ of PR-type over $\Bdaggerrigk$.
\begin{defi}
  \label{deficomplesol}
  Let $g \in \Dinftyg(D) $ and $r$ be big enough such that $\Dinftyg(D)$ 
  admits the description in \eqref{dinftyg}. A family of elements 
  $(G_k)_{k \in \mb{Z}}$ in $\Bdaggerlogk \otimes_{\Bdaggerrigk} \Ndr(D)$ is 
  called a \textbf{complete solution} for $(1 - \varphi)G = g$ if   
  $\partial(G_k) = G_{k+1}$ (cf. \ref{deltasurjd}) and
  $\partial^r (g) = (1 - p^r \varphi)G_r $ for $r$ big enough.
\end{defi}
If $G = (G_k)$ is a complete solution of $g \in \Dinftyg(D)$  we also write 
$\partial^{-k}(G) = G_k$ by abuse of notation. Let $s \gg 0$ such that 
$(1 - p^s \varphi) G_s = \partial^s(g)$. Then one sees inductively
thanks to Lemma \ref{lemndrpsi1} that 
\begin{align*}
     N(G_k) &= \sum_{j \geq -k} \lambda_{j} \frac{t^{j+k}}{(j+k)!} =: L_k,  ~~ 
     \lambda_{j} \in \Dst^K(D)\\
     (\psi \otimes 1 - p^{-k} \otimes \varphi)(G_k) 
     &= p^{-k} \sum_{j \geq -k}  \mu_{j} \frac{t^{j+k}}{(j+k)!} 
     =: (\psi \otimes 1)(M_k), ~~ \mu_{j} \in  \Dst^K(D),
\end{align*}
where for almost all $j$ one has $\lambda_j = \mu_j = 0$. 
On $\Bdaggerlogk \otimes_{K_0} \Dst^K(D)$, as one checks easily,  
we have the identity of operators
\[
  (pN \otimes 1 + 1 \otimes N) (\psi \otimes 1 - p^{-k} \otimes \varphi) = 
  (\psi \otimes 1 - p^{-k+1} \otimes \varphi)(N \otimes 1 + 1 \otimes N) = 
  (\psi \otimes 1 - p^{-k+1} \otimes \varphi)N,
\]
hence 
\[
  N ( (\psi \otimes 1)(M_{k})) = (\psi \otimes 1 - p^{-k+1} \otimes \varphi)
  (L_{k}),
\]
since $N \otimes 1$ vanishes on elements of $\sum t^i \cdot
\Dst^{K}(D)$, hence the relation (by applying $(\psi^{-1} \otimes 1 =
\varphi \otimes 1$, which we may since $\psi$ acts invertibly on $\sum
t^i \cdot \Dst^{K}(D)$)
\[
  N(M_{k}) =  (1 - p^{-k+1} \varphi)(L_{k}).
\]
On the coefficients this implies the relation
\[
  N \mu_{j} = (1 - p^{-j+1}\varphi) \lambda_{j}.
\]
If $A = \sum_{j \geq -k} \nu_{j} / (j+k)! \cdot t^{j+k}$ and if one
changes $G_k$ to $G'_k = G_k + A$ so that still $\partial^k(G'_k) =
\partial^k(G_k)$, then $\lambda_{j}$ is changed to $\lambda_{j} +
N(\nu_{j})$ and $\mu_{j}$ is changed to $\mu_{j} + (1 - p^{j} \varphi)
\nu_{j}$. Hence, $L_k$ is changed to $L_k + N(A)$ and $M_k$ is changed
to $M_k + (1 - \varphi)(A)$, so that the class of $(\lambda_i, \mu_i)$
is well-defined in $H^1(\mf{C}(D|_L(-i)))$.  The tupel $(\lambda_j,
\mu_j)$ may be considered as an element of $H^1(\mf{C}(D|_L(-i)))(i)$,
and we denote the collection of these elements element by
$\mc{R}_D(g)$, i.e. one has a $\Gamma_K$-equivariant map
\begin{equation}
  \mc{R}_D: \Dinftyg(D) \lra \bigoplus_{i \in \mb{Z}} H^1(\mf{C}(D|_L(-i)))(i). 
  \label{mcrd}
\end{equation}
We note that the map
$\widetilde{\mc{R}}_D$\index{RtildeD@$\widetilde{\mc{R}}_D$} in Lemma
\ref{lemndrpsi1} is the composition of $(1 - \varphi)$ with $\mc{R}_D$
and the natural projection to the sum $\bigoplus_{-k \leq i < 0}
H^1(\mf{C}(D|_L(-i)))(i)$.

Define for all $k \in \mb{Z}$
\begin{align*}
  N(G_k) &= L_k =: \partial^{-k}(L) \\
  (\psi \otimes 1 - 1 \otimes \varphi) (G_k) &= \psi \otimes 1 (M_k) =: 
  \psi \otimes 1 (\partial^{-k}(M)).
\end{align*}
These definitions imply that (calculating again in $\Bdaggerlogk \otimes_{K_0} 
\Dst^K(D)$)
\[
  \psi((1 - \varphi)(G_k) - M_k) = (\psi \otimes 1)((1 - \varphi)(G_k) - M_k)) 
  = 0,
\]
hence, since $\partial$ acts invertibly on $(\Bdaggerlogk \otimes_{K_0} 
\Dst^K(D))^{\psi = 0}$, 
\[
  \partial^k(g) = (1 - p^k \varphi) G_k - M_k. 
\]
Of course, $M_k = L_k = 0$ for $k$ big enough. We will also refer to
the system $H = (L_k^{[1]}, M_k, G_k)$ as a \textbf{complete solution}
for $g \in \Dinftyg(D)$, where by $L_k^{[1]}$ we mean that the action
of $\varphi$ is multiplied by $p$. This extra factor is introduced so
that the interpolation property holds.

Following Perrin-Riou, we set
\[
  \mc{U}(D) := \bigoplus_{i \in \mb{Z}} t^i \cdot \Dst(D)
\]\index{UD@$\mc{U}(D)$}
and
\[
  \Dinftyg^2(D) := \mc{U}(D) / (1-p\varphi, N) \mc{U}(D).
\]\index{Dinftygzwei@$\Dinftyg^2(D)$}

\begin{prop}
  \label{propdinfty}
  One has the following exact sequences of $\Gamma_K$-modules:
  \begin{gather*}
    0 \lra \Dinftye(D) \lra \Dinftyg(D) \overset{\mc{R}_D}{\lra} 
    \bigoplus_{i \in \mb{Z}} H^1(\mf{C}_L(D|_L(-i)))(i) \\
    0 \lra \Dinftyf(D) \lra \Dinftyg(D) \overset{\overline{\mc{R}}_D}{\lra}
    (\mc{U}(D|_L) / N \mc{U}(D|_L))^{\varphi = p^{-1}} \\
    0 \lra \Dinftye(D) \lra \Dinftyf(D) \overset{\mc{R}_D}{\lra}
    (\mc{U}(D|_L))^{N = 0} / 
    (1 - \varphi) (\mc{U}(D))^{N = 0}.
  \end{gather*}
\end{prop}
\begin{proof}
  See \cite{perrin01}, Proposition 2.3.4.
\end{proof}

We remark that in the case where $K / \mb{Q}_p$ is unramified one can 
show all the right-most maps in the preceeding Proposition are actually 
surjective. This can be deduced as in \cite{perrin01}, Proposition 4.1.1.
Additionally, using the preceding Proposition, one can show that 
$\Dinftyf(~)$ need not be exact.

%

\begin{defi}
  \begin{enumerate}
  \item For a torsion free element $\gamma$ of $\Gamma_K$ and $i \in \mb{Z}$
    Perrin-Riou's differential operator $\nabla_i = l_i$
    \index{nablai@$\nabla_i$} is defined as
    \[
    \nabla_i = \frac{\log(\gamma)}{\log_p (\chi(\gamma))} - i = \nabla_0 - i
    \]
  \item The operator $\nabla_0 / (\gamma_n - 1)$
    \index{nabla0gamma@$\nabla_0 / (\gamma_n - 1)$} for $n$ such that 
    $\Gamma_n$ is cyclic is defined as
    \[
    \frac{\nabla_0}{\gamma_n - 1} := \frac{\log
      (\gamma_n)}{\log_p (\chi(\gamma)) (\gamma_n - 1)} := 
    \frac{1}{\log_p (\chi(\gamma_n))} \sum_{i = 1}^\infty 
    \frac{(1-\gamma_n)^{i-1}}{i}.
    \]
  \end{enumerate}
\end{defi}
First, we remark that the second operator is \textit{not} a quotient
of two operators, although it behaves as one would like. To clarify we
observe that the first definition is independent of the choice of
$\gamma$ since $\log(\gamma^m) / \log_p (\chi(\gamma^m)) = m/m \cdot
\log(\gamma) / \log_p (\chi(\gamma))$. Hence, if $\nabla_0(y)$ for some
$y \in D$ (for instance, $y \in D^{\psi = 0}$) is such that $\gamma_n
- 1$ acts invertibly on it we see that $(\gamma_n - 1)^{-1} \nabla_0
(y) = \frac{\nabla_0}{\gamma_n-1} (y)$. From this it also follows that
$(\gamma_n - 1) \frac{\nabla_0}{\gamma_n-1} = \nabla_0$.  Secondly we
observe that
\[
  \nabla_i = \frac{ \log(\chi(\gamma)^{-i} \cdot \gamma)}{\log_p (\chi(\gamma))} =
  \mbox{Tw}^{-i}\left( \frac{ \log (\gamma)}{\log_p 
  (\chi(\gamma))}  \right)
\]
where $\mbox{Tw}^i$ is the operator on $\mc{B}(\Gamma_K)$ which sends
$\gamma$ to $\chi(\gamma)^k \gamma$.

\begin{defi}
If $h \geq 1$ we define $\Omega_h := \nabla_{h-1} \cdot \ldots \cdot
\nabla_0 \in \mc{H}(\Gamma_K)$\index{Omegah@$\Omega_h$}.
\end{defi}

\begin{lem}
  \label{lemomegandr}
  Let $D$ be a de Rham $(\varphi, \Gamma_K)$-module over
  $\Bdaggerrigk$ and let $h \in \mb{N}$ such that $\Fil^{-h} \Ddr^K(D)
  = \Ddr^K(D)$. Then $\Omega_h(\Ndr(D)) \subset D$.
\end{lem}
\begin{proof}
  Since $\Omega_h = \nabla_{h-1} \circ \nabla_{h-2} \circ \ldots \circ
  \nabla_{0} = t^h \partial^h$ it suffices to show that $t^h \Ndr(D)
  \subset D$.  First assume that $D$ is semi-stable. We know from
  Proposition \ref{propdpositive} that if $D$ is positive, then
  $\Dst^K(D) = (\Bdaggerlogk[1/t] \otimes D)^{\Gamma_K} \subset
  \Bdaggerlogk \otimes_{\Bdaggerrigk} D$, so that $\Ndr(D) =
  (\Bdaggerlogk \otimes_{\Bdaggerrigk} \Dst^K(D))^{N = 0} \subset
  D$. For general $D$ if $h \geq 1$ is as in the statement then
  $D(-h)$ is positive, so that $t^h \Ndr(D) \subset D$. Now if $D$ is
  de Rham and $L/K$ a finite extension such that $D|_L$ is
  semi-stable, then we have that $t^h \Ndr(D) \subset t^h \Ndr(D|_L)
  \subset D|_L$ and $t^h \Ndr(D) \subset D[1/t]$, so that $t^h \Ndr(D)
  \subset D$ as required.
\end{proof}

\begin{defi}
  Let $D$ be a de Rham $(\varphi, \Gamma_K)$-module over $\Bdaggerrigk$ and
  $h \geq 1$ be such that $\Fil^{-h} \Ddr^K(D) = \Ddr^K(D)$. 
    We define Perrin-Riou's \textbf{big exponential map} by 
    \begin{align*}
      \Omega_{D,h}: \Dinftyg(D) &\longrightarrow D^{\psi = 0} \\ g
      &\longmapsto \nabla_{h-1} \circ \ldots \circ \nabla_0 (g)
    \end{align*}
\end{defi}

\begin{lem}
  One has the following commutative diagram:
  \[
  \xymatrix{ \Dinftyg(D) \ar[r]^{\partial^{-k}} \ar[d]^{\Omega_h} &
    \Dinftyg(D(k)) \ar[d]^{\Omega_{h+k}} \\ D^{\psi = 0} \ar[r]^{t^k}
    & D(k)^{\psi = 0} }
  \]
\end{lem}
\begin{proof}
  This is clear from the fact that $\Omega_h = t^h \partial^h$.
\end{proof}



\begin{lem}
  Let $D$ be as before and assume that $K$ is such that $\Gamma_{K}$
  is torsion free. Then one has a canonical map $h^1_{K,
      D} : (\varphi - 1) D^{\psi = 1} \ra H^1(K, D) / (D^{\varphi = 1} /
  (\gamma_K - 1))$ such that the diagram
  \[
  \xymatrix@C=5.0pc{
    (\varphi - 1) D^{\psi = 1 } 
    \ar[d]_{\tilde{h}^1_{K,D}} & 
    \ar@{->>}[l]_{~~~~~ \varphi - 1} D^{\psi = 1}  \ar[d]^{h^1_{K, D}} \\
    H^1(K, D) / (D^{\varphi = 1} / (\gamma_{K} - 1)) & \ar@{->>}[l] H^1(K, D)
  }
  \]
  is commutative.
\end{lem}
\begin{proof}
  Obviously $D^{\psi = 1} / D^{\varphi = 1} \cong (\varphi - 1)D^{\psi
    = 1}$.  It is clear that the map $h^1_{K_n, D}$ factorizes over
  $D^{\psi = 1}_{\Gamma_K}$. The claim follows.
\end{proof}

\begin{rem}
  If $D$ is of PR-type and let $h$ be such that (\ref{dinftyg}) is
  satisfied. If $g \in \Dinftyg(V)$ and $k \geq 1-h$ we actually
  have $\Omega_{h}(g) \otimes e_k \in (1 - \varphi) D(k)^{\psi = 1}$.
\end{rem}	
\begin{proof}
  Let $\partial^{-k}(g) = (1 - \varphi)\partial^{-k}(G) - \partial^{-k}(M)$. 
  Then 
  \[
    \partial^{-k}(M) = \sum_{j \geq 0}^{h+k-1} \mu_{j-k}
    \frac{t^{j}}{j!} \in \mc{H} \otimes \Dst(V(k)).
  \]
  Since $\nabla_{h+k-1} \circ \ldots \circ \nabla_0 = t^{h+k} \partial^{h+k}$
  the $\partial^{-k}(M)$-part of $\partial^{-k}(g)$ is killed by $\Omega_{h}$.
\end{proof}

Hence, we see that if $h$ is such that (\ref{dinftyg}) is satisfied and
$h - r > 0$ the diagram 
\[
\xymatrix@C=2.5pc{ 
  (\Bdaggerlogk \otimes_F \Dst^K(D(-r))^{N = 0, \psi = 1} \ar[r]^-{\Omega_{h-r}}
  \ar[d]^{1-p^r\varphi} & D(-r)^{\psi = 1} \ar[d]^{1-p^r \varphi}   \\ 
  (1 - p^r \varphi)(\Bdaggerlogk \otimes_F \Dst^K(D(-r))^{N = 0, \psi = 1} 
  \ar[r]^-{\Omega_{h-r}}
  \ar[d]^{\partial^{-r}} & (1 - p^r \varphi) D(r)^{\psi = 1} 
  \ar[d]^{\textrm{Tw}^r}  \\ 
  \Dinftyg(D) \ar[r]^{\Omega_h} & (1 - \varphi) D^{\psi = 1}
  }
\]
commutes.

Let $D$ be of PR-type, $g \in \Dinftyg(D)$ and $G = (L_k, M_k, G_k)$
be a complete solution for $g$. Then for each $k$ and $n \gg 0$ one
has that the element
\[
\!\!\!\!  \Xi_{n,k}(G) := p^{n(k-1)} \varphi^{-n} \partial^{-k}(H)(0) := 
  p^{n(k-1)}(p^{-n}  \varphi^{-n} \partial^{-k}(L)(0), 
   \varphi^{-n} \partial^{-k}(M)(0), 
   \varphi^{-n} \partial^{-k}(G)(0))
\]\index{Xink@$\Xi_{n,k}(G)$}
may be viewed as an element in $H^1(\mf{C}_{\textrm{st}}(K, D(k)))$ (see
(\ref{eqh1mfc})).

\begin{thm}
  \label{thmsecond}
  Let $D$ be a de Rham $(\varphi, \Gamma_K)$-module over $\Bdaggerrigk$, 
  $g \in \Dinftyg(D)$ and $G$ a complete solution for $g$ in $L$. Let
  $h$ be such that (\ref{dinftyg}) is satisfied.
  Then for $k \geq 1 - h$ and $n \gg 1$ one has 
  \begin{align*}
    h^1_{K_n, D(k)}(\nabla_{h-1} &\circ \ldots \circ \nabla_0(g)
    \otimes e_k) \\
    &= p^{-n(K_n)} (-1)^{h+k-1}(h+1-k)!
    \frac{1}{[L_n:K_n]}\textrm{Cor}_{L_n/K_n}\exp_{K_n, D(k)}(\Xi_{n,
      k}(G)),
  \end{align*}
  where we consider the elements on both sides in $H^1(K_n, D) / 
  (D^{\varphi = 1} / (\gamma_{K_n} - 1))$.
\end{thm}
\begin{proof}
  The proof is divided into several parts. The first general assumption is 
  that $D$ is of PR-type. 

  Let $D$ be pure of slope $\leq 0$. Then the exponential map has the 
  description given in  Proposition \ref{propbkeins}. We may assume 
  $n$ big enough so that $\Gamma^n_K$ is torsion free. Recall the 
  relation
  \[
  \Omega_{D(k), h+k}(\partial^{-k} (G)) = \Omega_{D,h}(G) \otimes e_k
  \] 
  Hence, for the $k \geq 1 - h$ we have
  \[
    h^1_{K_n, D(k)}(\nabla_{h-1} \circ \ldots \circ \nabla_0 (G) \otimes e_k) 
    = h^1_{K_n, D(k)}(\nabla_{h+k-1} \circ \ldots \circ \nabla_0 (\partial^{-k} (G))).
  \]
  Let $y_h = \nabla_{h+k-1} \circ \ldots \nabla_0 (\partial^{-k}(G))$ and
  $w_{n,h} = \nabla_{h+k-1} \circ \ldots \frac{\nabla_0}{\gamma_n - 1} 
  (\partial^{-k}(G))$. Then in this case
  \[
    h^1_{K_n, D(k)}(y_h)(\sigma) = \frac{ \sigma - 1 }{ \gamma_n - 1 }  
    y_h - (\sigma - 1) b_{n, h} \in H^1(K_n, D(k)),
  \]
  where $b_{n,h} \in \TD$ is such that $(\gamma_n - 1)(\varphi - 1)b_{n,h} = 
  (\varphi - 1) y_h$. Recall that 
  $\partial^{-k}(g) = (1 - \varphi)\partial^{-k}(G) - \partial^{-k}(M)$ and 
  $\Omega_{D(k), h+k}(\partial^{-k}(g)) = (1 - \varphi)
  \Omega_{D(k), h+k}(\partial^{-k}(G))$, hence
  \[
    \nabla_{h+k-1} \circ \ldots \circ \frac{\nabla_0}{\gamma_n - 1} (\partial^{-k}(g))
    = \nabla_{h+k-1} \circ \ldots \circ \frac{\nabla_0}{\gamma_n - 1} 
    ((1- \varphi)G_{-k})
    - \nabla_{h+k-1} \circ \ldots \circ \frac{\nabla_0}{\gamma_n - 1} 
    (M_{-k}).
  \]	
  With this we may choose
  \[
    b_{n,h} = (\varphi - 1)^{-1} \left( \frac{\Omega_{D(k), h+k}}{\gamma_n - 1}
    ((1- \varphi)G_{-k}) - 
    \frac{\Omega_{D(k), h+k}}{\gamma_n - 1} (M_{-k}) \right) \in \TD.
  \]
  Now for $n \gg 0$ we have $g \in \Bdaggernlogk \otimes \Dst^K(D)$, hence
  the cocycle $h^1_{K_n, V(k)}(y_h)(\sigma) = (\sigma - 1) \left(  
  w_{n,h} - b_{n,h}\right)$ is cohomologuous to
  \[
  h^1_{K_n, V(k)}(y_h)(\sigma) = (\sigma - 1) \left( 
  \varphi^{-n}(w_{n,h}) - \varphi^{-n}(b_{n,h}) \right)
  \]
  since $(\varphi - 1)(w_{n,h} - b_{n,h}) \in \Dst^K(D(k))$ so that
  $G_K$ acts trivially (and $\varphi$ acts as usual invertibly on
  $\Dst^K(D(k))$).  We use the exact sequences from the generalized
  Bloch-Kato map from Proposition \ref{propbkeins}.  By the general
  properties of the connecting homomorphism for continuous cohomology
  we have the following: if $(x, y, z) \in H^1(\mf{C}_{\textrm{st}}(K,
  D(k)))$ and $\tilde{x} \in \TDlog [1/t]$ is such that $g(\tilde{x})
  = (x,y,z)$ then $\exp_{K_n, D(k)}((x,y,z))(\sigma) = (\sigma - 1)
  \tilde{x}$.  First one has
  \[
    \varphi^{-n}(y) - \varphi^{-n}(y)(0) \in tK_0[[t]] \otimes_{K_0} \Dst^K(D),
  \]
  hence
  \[
    \frac{\nabla_0}{\gamma_n - 1} \varphi^{-n}(y) = p^{-n}
    \varphi^{-n}(y)(0) + t z_1.
  \]
  The same recursion as in \cite{berger02}, Theorem II.3 shows that 
  \[
    \varphi^{-n}(w_{n,h}) - (-1)^{h-1}(h-1)! p^{-n} \varphi^{-n}(y)(0) \in 
    \Bplusdr \otimes D.
  \]
  Next we have
  \[
    N(\varphi^{-n}(w_{n,h}) - \varphi^{-n}(b_{n,h})) = 
    p^{-n} \varphi^{-n}(
    \nabla_{h+k-1} \circ \ldots \frac{\nabla_0}{\gamma_n - 1} (N \partial^{-k}(G))).
  \]
  Again we see by recursion with our choice of $h$ that 
  since $N \partial^{-k}(G) = L_{-k}$ and 
  \[
     L_{-k} = \sum_{i = 0}^{h-1} \lambda_i \cdot t^{i} / i!,
  \]
  that we obtain an equality 
  \[
    p^{-n} \varphi^{-n}(
    \nabla_{h+k-1} \circ \ldots \frac{\nabla_0}{\gamma_n - 1} (L_{-k})) = 
    (-1)^{h-1}(h-1)! p^{-2n}  \varphi^{-n}(L_{-k})(0).
  \]
  Finally one has
  \[
    (\varphi - 1)(\varphi^{-n}(w_{n,h}) - \varphi^{-n}(b_{n,h})) = 
    \varphi^{-n}(
    \nabla_{h+k-1} \circ \ldots \frac{\nabla_0}{\gamma_n - 1} (M_{-k})).
  \]
  Similarly, as before we have
  \[
     M_{-k} = \sum_{i = 0}^{h-1} \mu_i \cdot t^{i} / i!, 
  \]
  so that the recursion shows 
  \[
    \varphi^{-n}(
    \nabla_{h+k-1} \circ \ldots \frac{\nabla_0}{\gamma_n - 1} (L_{-k})) = 
    (-1)^{h-1}(h-1)! p^{-n} \varphi^{-n}(M_{-k})(0).
  \]
  Altogether this shows that 
  \[
    (-1)^{h-1}(h-1)! p^{-n} \exp_{K_n, D(k)}(\Xi_{n, k}(G))(\sigma) = 
    (\sigma - 1)(\varphi^{-n}(w_{n,h}) - \varphi^{-n}(b_{n,h})),
  \]
  which is the claim in this case.


  Next assume $D$ is pure of slope $> 0$. Then the exponential map has the 
  description given in Proposition \ref{propbkzwei}. First we note that 
  $h^1_{K_n, D(k)}(\Omega_{D, h}(g) \otimes e_k)  = (x, y)$ with
  \[
    y = \Omega_{D(k), h+k}(G_{-k}),~~~ x = 
    \nabla_{h+k-1} \circ \ldots \circ \frac{\nabla_{0}}{\gamma_K - 1}(
    (\varphi - 1)(G_{-k})).
  \]
  The exponential map sends $\Xi_{n,k}(G)$ to $\varphi^{-n}(G_{-k})(0) \in 
  X^1(\TD)^{G_K}$. The identification $\TD / (\varphi -1) \overset{\sim}{\ra}
  X^1(\TD)$ is given by the following construction (see \cite{berger08a}, 
  Remark 3.4):
  If $x \in \TD / (\varphi -1)$ and $y \in \TD[1/t]$ is chosen so that 
  $(\varphi - 1)y = x$ then for $n \gg 0$ the image of $x$ is 
  $\varphi^{-n}(y)$. With this we see that under these identifications 
  the class of $h^1_{K_n, D(k)}(\Omega_{D, h}(g) \otimes e_k)$ is send to 
  \[
    \varphi^{-n}(\nabla_{h+k-1} \circ \ldots \circ
    \frac{\nabla_{0}}{\gamma_K - 1}(G_{-k})) \equiv (-1)^{h-1}(h-1)!
    p^{-n} \varphi^{-n}(G_{-k})(0) \mod \Bplusdr \otimes D
  \]
  where we use the same recursion as before, hence the claim in this case.

  In the general case of semistable a $D$ of PR-type one may use the exact 
  $0 \ra D_{\leq 0} \ra D \ra D_{> 0} \ra 0$, where $\D_{\leq 0}$ is the
  biggest submodule of $D$ with slopes $\leq 0$, and $D_{> 0} = D / D_{\leq 0}$,
  which is a $(\varphi, \Gamma_K)$-module with slopes $> 0$.
  By using the description of the isomorphism (\ref{eqheins}) and
  the explicit description of the transition morphism for 
  the cone one is reduced, since all maps are compatible with exact sequences, 
  to the case of a module with all slopes $\leq 0$ or all slopes $>0$.
  But in these cases we have just verified that the statement holds.

  
  Now assume $D$ is de Rham and let $L/K$ be a finite extension such that
  $D$ is of PR-type over $L$.
  Then for $y \in \Dinftyg(D)$ one has, if we consider 
  $y \in \Dinftyg(D|_L)$
  \[
    \textrm{Res}_{L_n / K_n}(h^1_{K_n, D(k)}(\Omega_{D, h}(y))) = 
    h^1_{L_n, D|_L(k)}(\Omega_{D, h}(y)), 
  \]
  so that the claim follows from Proposition \ref{proprescor}.
\end{proof}

For the record we state the next proposition in case $D$ is
semi-stable. As before, let $h \geq 1$ be such that
  $(\ref{dinftyg})$ is satisfied for $D$, and dually let $h^* \geq 1$
  be such that $(\ref{dinftyg})$ is satisfied for $D^*(1)$ 
\begin{prop}
  \label{propreci}
  \begin{enumerate}
    \item If $k \geq 1 - h$ and $n \geq 1$ then 
      \[
        h^1_{K_n, D(k)}(\nabla_{h-1} \circ \ldots \circ \nabla_0(g)
        \otimes e_k) = p^{-n(K_n)} (-1)^{h+k-1}(h+1-k)!  \exp_{K_n,
          D(k)}(\Xi_{n, k}(G))
      \]
    \item If $k \leq -h^*$ and $n \geq 1$ then
      \[
        \exp^*_{K_n, D^*(1)}(h^1_{K_n, D(k)}(\nabla_{h-1} \circ \ldots
        \circ \nabla_0(g) \otimes e_k)) = p^{-n(K_n)} 
        \frac{1}{(-h-k)!} \varphi^{-n}(\partial^{-k} g \otimes
        t^{-j} e_{j})(0)
      \]
  \end{enumerate}
\end{prop}
\begin{proof}
  The first part is just the preceding theorem. For the second observe
  that due to Proposition \ref{propdualexp} one has
  \[
    \exp^*_{K_n, D^*(1)}(h^1_{K_n, D(k)}(\nabla_{h-1} \circ \ldots
    \circ \nabla_0(g) \otimes e_k)) = p^{-n(K_n)} \varphi^{-n}
    (\nabla_{h-1} \circ \ldots \circ \nabla_0(g) \otimes e_k)(0).
  \]
  A computation with the Taylor series shows that 
  \[
    p^{-n(K_n)} \varphi^{-n} (\nabla_{h-1} \circ \ldots \circ
    \nabla_0(g) \otimes e_k)(0) = p^{-n(K_n)} \frac{1}{(-h-j)!} 
    \varphi^{-n}(\partial^{-k} g \otimes t^{-k} e_k)(0),
  \]
  hence the claim.
\end{proof}

\clearpage


\bibliography{diss3} \bibliographystyle{plain}

\rule{1in}{.3mm}

\small 
Andreas Riedel

Universit\"at Heidelberg, Mathematisches Institut

Im Neuenheimer Feld, 228

69115 Heidelberg, Germany

\vspace{2mm}

Email: \texttt{ariedel@mathi.uni-heidelberg.de}

Phone: 0049 6221 545690

\end{document}